\documentclass[12pt]{article}
\usepackage[a4paper,margin=2.5cm]{geometry}

\usepackage{amsmath}
\usepackage{amsthm}
\usepackage{amssymb}

\usepackage{tikz}
\usetikzlibrary{arrows.meta, decorations.markings}

\tikzset{
  mid arrow/.style={
    postaction={decorate,
      decoration={markings,
        mark=at position #1 with {\arrow[scale=1.3]{Stealth}}
      }
    }
  },
  mid arrow/.default=0.55
}

\usepackage[labelsep=newline]{caption}
\usepackage{subcaption}

\usepackage[hidelinks]{hyperref}
\hypersetup{
  pdftitle={A New Invariant for Doubly Periodic Solutions of the Sinh-Gordon Equation},
  pdfauthor={Nicolas Hasse, Ross Ogilvie, Martin U. Schmidt},
  pdfsubject={Mathematics Subject Classification (2020): 14H15, 14H70, 37K20, 53A10, 30F30, 20H05},
  pdfkeywords={spectral curves, Whitham deformations, isoperiodic sets, constant mean curvature tori, sinh-Gordon equation, Willmore energy},
}

\newcommand{\labelthis}[1]{\stepcounter{equation}\tag{\theequation}\label{#1}}

\numberwithin{equation}{section}

\newtheorem{theorem}{Theorem}[section]
\newtheorem{lemma}[theorem]{Lemma}

\theoremstyle{definition}
\newtheorem{definition}[theorem]{Definition}

\newtheorem{example}[theorem]{Example}

\DeclareMathOperator{\Real}{Re}
\DeclareMathOperator{\Imag}{Im}
\DeclareMathOperator*{\Res}{Res}

\DeclareMathOperator{\id}{id}

\DeclareMathOperator{\Hom}{Hom}

\DeclareMathOperator{\diag}{diag}

\DeclareMathOperator{\Mat}{Mat}
\DeclareMathOperator{\GL}{GL}
\DeclareMathOperator{\SL}{SL}
\DeclareMathOperator{\PSL}{PSL}

\DeclareMathOperator{\Sp}{Sp}
\DeclareMathOperator{\Sym}{Sym}
\DeclareMathOperator{\Perm}{Perm}

\DeclareMathOperator{\Gr}{Gr}
\newcommand{\CP}{\bbC\bbP}
\newcommand{\RP}{\bbR\bbP}

\DeclareMathOperator{\Per}{Per}
\DeclareMathOperator{\APer}{APer}
\DeclareMathOperator{\BPer}{BPer}
\DeclareMathOperator{\CA}{CA}
\newcommand{\willmore}{\calW}
\newcommand{\realPoly}{\calP_\bbR}

\newcommand{\ci}{i}
\newcommand{\dd}{\mathrm{d}}

\newcommand{\bbC}{\mathbb{C}}

\newcommand{\bbF}{\mathbb{F}}

\newcommand{\bbH}{\mathbb{H}}

\newcommand{\bbN}{\mathbb{N}}

\newcommand{\bbP}{\mathbb{P}}
\newcommand{\bbQ}{\mathbb{Q}}
\newcommand{\bbR}{\mathbb{R}}
\newcommand{\bbS}{\mathbb{S}}

\newcommand{\bbZ}{\mathbb{Z}}

\newcommand{\calB}{\mathcal{B}}

\newcommand{\calF}{\mathcal{F}}

\newcommand{\calH}{\mathcal{H}}

\newcommand{\calP}{\mathcal{P}}

\newcommand{\calS}{\mathcal{S}}

\newcommand{\calW}{\mathcal{W}}

\parskip = \baselineskip

\title{A New Invariant for Doubly Periodic Solutions of the Sinh-Gordon Equation}
\author{Nicolas Hasse, Ross Ogilvie, Martin U.\ Schmidt}
\date{}

\begin{document}
\maketitle
\begin{NoHyper}\renewcommand{\thefootnote}{}\footnotetext{Date: \today.\\Mathematics Subject Classification (2020): Primary 14H15; Secondary 14H70, 37K20, 53A10, 30F30, 20H05.}\end{NoHyper}
\begin{abstract}
Spectral curves of constant mean curvature tori in Euclidean three-space are real hyperelliptic curves carrying two meromorphic differentials with integral periods.
These transcendental period conditions are studied through Whitham deformations, which deform the curve while preserving the periods.
We extend Whitham deformations to all real hyperelliptic curves of this type, by holding fixed the plane spanned by the periods, or this plane together with its complex structure.
Under mild nondegeneracy conditions the resulting isoperiodic sets are manifolds of dimensions three and one, respectively.
Along the one-dimensional isoperiodic sets, a generalization of the Willmore energy is a local parameter away from its critical points, to which the spectral curves of tori belong.
Finally, using homology bases adapted to the branch points, we find discrete invariants of the isoperiodic sets of doubly periodic curves: a triple of integers and a refined conformal type.
\end{abstract}

\section{Introduction}

Wente's discovery of immersed constant mean curvature tori in $\bbR^3$~\cite{Wente1986} led to their classification in terms of spectral data by Pinkall and Sterling~\cite{Pinkall1989}, Hitchin~\cite{Hitchin1990} and Bobenko~\cite{Bobenko1991}.
CMC surfaces are a large and rich class of surfaces~\cite{Kapouleas1995}, and are some of the principal examples of constrained Willmore surfaces, which makes them candidates for minimizers of the Willmore energy in their conformal class~\cite{Heller2018a}.
The induced metric of CMC surfaces satisfies the sinh-Gordon equation, an integrable system.
The spectral data of a solution to the sinh-Gordon equation consist of a real hyperelliptic curve $\Sigma_a$, a pair of meromorphic differentials, and a line bundle.
The doubly periodic solutions of the sinh-Gordon equation are characterized by the condition that the differentials have periods in $2\pi\ci\bbZ$ (a further condition is needed to ensure that the immersion itself is doubly periodic).

One therefore studies the spectral data of CMC tori within a set of real hyperelliptic curves $\calH^g$.
The shape of this subset, known only for some low-genus cases~\cite{Carberry2020}, is intricate.
Even the existence of CMC tori of each spectral genus is nontrivial~\cite{Ercolani1993,Jaggy1994}, because the conditions are given by abelian integrals.
As the spectral curves of CMC tori in $\bbR^3$ are believed to comprise a totally disconnected subset of $\calH^g$, one studies their closure.
Their closure lies in the set $\calS^g$ where the differentials have a common root on $\bbS^1$~\cite{Carberry2016}, and is known completely for $g=2$~\cite{Carberry2023arxiv}, whereas for CMC tori in $\bbS^3$ it is all of $\calH^g$~\cite{Carberry2016a}.
The main tool for these results is the Whitham deformation, or isoperiodic flow, which deforms the algebraic curve while preserving the periods of the differentials~\cite{Grinevich1995,Kilian2015,Carberry2019}.

In this paper we take a different approach, in that we do not try to stay within the closure of spectral curves of tori, but instead generalize the tools for Whitham deformations to all nonsingular curves of $\calH^g$.
The natural extension of isoperiodic flow beyond the doubly periodic solutions of the sinh-Gordon equation is to preserve the periods of the whole plane of differentials with purely imaginary periods.
This reformulates the isoperiodic condition as holding constant a point in the Grassmannian $\Gr(2,\bbR^g)$.
We also show how to give this plane of periods a complex structure, making it a point in $\CP^{g-1}$, and consider the Whitham deformations that additionally preserve this complex structure.
These two kinds of Whitham deformation yield isoperiodic sets of dimensions $3$ and $1$, respectively (Theorem~\ref{thm:Whitham deformation}).
For doubly periodic solutions there is a formula for the Willmore energy of a fundamental domain that depends only on the spectral data~\cite[Theorem~5.7]{Knopf2025}.
Dividing out the area of the period parallelogram, we obtain the \emph{normalized Willmore energy} $\willmore$, which is defined for every nonsingular curve of $\calH^g$, not only in the doubly periodic case.
We show that $\willmore$ is a local parameter of the $1$-dimensional isoperiodic sets away from the set $\calS^g$, which contains the closure of the spectral curves of tori, and that it has a critical point at each curve of $\calS^g$ (Theorem~\ref{thm:Willmore derivative}).

In the final section, we present a new set of invariants for the isoperiodic sets in the doubly periodic case.
For the two types of Whitham deformations, these are respectively a triple of integers $(n_1,n_2,n_3)$ (Theorem~\ref{thm:CAg orbits of Pi}), and a point $\tau \in \calF_1, \calF_3$, or $\calF_6$, depending on how many $n_i$ coincide, where $\calF_k$ are quotients of the upper half-plane by subgroups of the modular group of index $1,3$, or $6$ (Theorem~\ref{thm:CAg orbits of τ}).
In future work we will show that these are complete invariants for isoperiodic sets in $\calH^2$, the genus $2$ case.
The discovery of these invariants came from trying to understand the observation in~\cite[Remark~4.2]{Knopf2025} that there were three examples per conformal class of lattice.
The key notion is that of \emph{adapted canonical bases of homology} (Definition~\ref{def:adapted basis}): unlike arbitrary canonical bases, they are tied to the pairs of branch points, so the period lattice expressed in such a basis retains information that period-preserving deformations cannot change.


During the preparation of this work, Claude (Anthropic) was used in the creation of figures, and in the checking and editing of this article.
In terms of ideas, it found the classification of orbits by triples in Section~\ref{sec:adapted bases}; all other ideas originate from the authors.
The statements of this paper have been formalized in the Lean~4 proof assistant, conditional on an axiomatic description of path integrals.

\subsection{The Definition of \texorpdfstring{$\calH^g$}{Hg}}
\label{sec:def Hg}
The spectral curves corresponding to real finite-type solutions of the sinh-Gordon equation with spectral genus $g$ are hyperelliptic complex curves $\Sigma = \Sigma_a$, described by the equation $\nu^2 = \lambda\,a(\lambda)$ for a polynomial $a$ of degree $2g$ without roots at $\lambda=0$, together with a real involution that has no fixed points and that exchanges the points over $\lambda = 0,\infty$.
By compactification, we regard $\Sigma$ as a compact, hyperelliptic surface over the Riemann sphere $\lambda \in \CP^1$, so that the hyperelliptic involution is
\begin{equation}
\label{eq:sigma}
\sigma: \Sigma \to \Sigma,\; (\lambda,\nu) \mapsto (\lambda,-\nu),
\end{equation}
and the real involution covers $\lambda \mapsto \bar{\lambda}^{-1}$.
We denote the part of $\Sigma$ that lies over $\bbC^\times$ by $\Sigma^\circ = \{(\lambda,\nu) \in \bbC^\times \times \bbC \mid \nu^2 = \lambda\,a(\lambda)\}$.
The composition of the real involution without fixed points and the hyperelliptic involution must be a real involution $\rho$ that fixes all points with $|\lambda|=1$. 
Consequently $a\in \realPoly^{2g}$, where for $m\in \bbN$ we define $\realPoly^m$ as the space of polynomials $p$ of degree at most $m$ that satisfy the reality condition
\[
\lambda^{m}\, \overline{p(\bar{\lambda}^{-1})} = p(\lambda).
\]
The involution $\rho$ therefore acts as
\begin{equation}
\label{eq:rho}
\rho : \Sigma\to\Sigma,\; (\lambda,\nu) \mapsto (\bar{\lambda}^{-1}, \bar{\lambda}^{-(g+1)} \,\bar{\nu}).
\end{equation}
Were $a$ to have a root over the unit circle, then this would be a fixed point not only of $\rho$, but also of the real involution without fixed points $\rho\circ\sigma$.
Therefore $a(\lambda)$ cannot be zero for $\lambda \in \bbS^1$ for tori.
This notion of spectral curve in~\cite{Hitchin1990} is too rigid when one considers deformation theory, as the limit of roots may indeed come to $\bbS^1$.
Since for any $a \in \realPoly^{2g}$ the function $\lambda^{-g}a(\lambda)$ takes real values on the unit circle, we instead require that it be nonnegative.

There is an action of $\bbR^\times$ on $\calP^{2g}_\bbR$ by scaling that produces isomorphic curves $\Sigma_a$, with which we may normalize so that $\lambda^{-g}a(\lambda)$ is positive for all unimodular $\lambda$ and that $a(0)$ is unimodular.
Further, there is a group action by $e^{\ci\phi}\in\bbS^1$ on $a\in\calP^{2g}_\bbR$, given by $a(\lambda) \mapsto e^{g \ci \phi}a(e^{-\ci \phi}\lambda)$, which is a rotation of $\nu$ by $e^{(g+1)\ci\phi/2}$, so the orbit of this action again contains polynomials $a$ with isomorphic $\Sigma_a$.
We reduce the group $\bbS^1$ acting on the polynomials $a$ to the discrete subgroup $\{e^{2\ci\pi l/g}\mid l\in\bbZ/g\bbZ\}$ by imposing the condition $a(0)=1$, which by the reality condition is equivalent to $a$ being monic.
This condition ensures that any hyperelliptic curve corresponds to a finite set of polynomials $a$.
Hence if we choose a polynomial $a$ to represent $\Sigma_a$ and then deform the hyperelliptic curve, there is a unique deformation of $a$, which is simpler than dealing with a quotient space of equivalent $a$.

We summarize our conditions on $a$ in the following definition:
\begin{definition}
\label{def:Hg}
For any $g\in\bbN_0$ we define $\calH^g$ as the subset of $a\in\bbC^{2g}[\lambda]$ that obey
\begin{enumerate}
\item[(H1)]
\emph{Reality condition:} $a\in\calP^{2g}_\bbR$.
\item[(H2)]
\emph{Positivity condition:} $\lambda^{-g}a(\lambda)\ge0$ for all unimodular $\lambda$.
\item[(H3)]
\emph{Monic condition:}
$a(0)=1$.
\end{enumerate}
We use $\calH^g_0$ to denote the subset where $a\in\calH^{g}$ has $2g$ distinct roots.
\end{definition}

The set $\calH^g_0$ is the subset of $a\in\calH^{g}$ with smooth curves $\Sigma_a$.
It is an open subset of the affine subspace of $a\in\calP^{2g}_\bbR$ with $a(0)=1$, which is $(2g-1)$-dimensional.
Furthermore, $\calH^g$ is the closure of $\calH^g_0$.
Observe that since $a \in \calH^g_0$ has no roots on the unit circle, its roots must occur in pairs $\alpha, \bar{\alpha}^{-1}$.
We will use $\alpha_i$ to refer to the $g$ roots of $a$ inside the unit circle.
Since the roots of $a$ are branch points, they correspond to a single point of $\Sigma_a$, which we will occasionally also refer to as $\alpha_i$.
For the quick construction of examples, it is convenient to have the form of $a$ in terms of its roots:
\[
a(\lambda)
= R \prod_{j=1}^g (\lambda - \alpha_j)(1-\bar{\alpha}_j\lambda),
\quad
R^{-1} = \prod_{j=1}^g |\alpha_j|,
\quad
\sum_{j=1}^g \arg \alpha_j = \pi g \mod 2\pi.
\]

\subsection{Meromorphic Differentials \texorpdfstring{$\calB_a$}{B\_a} and a Period Map}
\label{sec:differentials}

Not all curves of $\calH^g_0$ are the spectral curves of doubly periodic solutions to the sinh-Gordon equation.
Such solutions correspond geometrically to constant mean curvature immersions of plane into $\bbR^3$ with a doubly periodic metric.
A curve $\Sigma_a$ is such a spectral curve if and only if there exist two meromorphic differentials $\Theta^1,\Theta^2$ such that $\Theta^1$ and $\Theta^2$ are the logarithmic derivatives $\dd\log\mu$ of the eigenvalues $\mu$ of the monodromy along two generators of the period lattice of the metric.
Equivalently,
\begin{enumerate}
\item[(P1)]
Their only poles are residue-free double poles at $\lambda = 0,\infty$, and they are $\bbR$-linearly independent.
\item[(P2)]
They have the symmetries
\begin{equation}
\label{eqn:Theta symmetries}
\sigma^* \Theta^k = -\Theta^k
\quad\text{and}\quad
\rho^* \Theta^k = -\overline{\Theta^k}\,.
\end{equation}
\item[(P3)]
Their periods lie in $2\pi\ci\bbZ$.
\end{enumerate}
Among the CMC immersions of a plane with double periodic metric are CMC immersions of a torus.
There are two further conditions on the differentials $\Theta^1, \Theta^2$ that distinguish this case: 
There is a point $\lambda_0$ (called the Sym point) such that
\begin{enumerate}
\item[(P4)]
They have a common root at $\lambda_0 \in \bbS^1$.
\item[(P5)]
Their integrals along a path between the two points of $\Sigma_a$ over $\lambda_0$ also lie in $2\pi\ci\bbZ$.
\end{enumerate}
In terms of eigenvalues of the monodromy, these conditions say, respectively, that $\dd\mu = 0$ and $\mu = \pm 1$ at $\lambda_0$.
The differentials that obey (P1)--(P5) are in general a subset of those that obey (P1)--(P3).
In such a case, several fundamental domains of the doubly periodic metric are required to cover the torus once.

Some of these conditions can be satisfied on every curve $\Sigma_a$ with $a \in \calH^g_0$.
For any $b \in \realPoly^{g+1}$ we define the differential form on $\Sigma_a$
\begin{equation}
\label{def:theta}
\Theta_b = \frac{b(\lambda)}{\lambda}\,\frac{\dd\lambda}{\nu} \; .
\end{equation}
On a hyperelliptic curve the differential $\nu^{-1} \dd\lambda$ is holomorphic ($g \geq 1$).
The factor $\lambda^{-1}$ and the degree of $b$ yield residue-free double poles at $\lambda = 0,\infty$.
Notice that both symmetry properties~\eqref{eqn:Theta symmetries} are satisfied.

\begin{definition}
\label{def:calBa}
Let $\calB_a$ be the $\bbR$-linear space of polynomials $b\in \realPoly^{g+1}$ whose corresponding differential form $\Theta_b$ has purely imaginary periods. 
\end{definition}

Certainly if the periods of $\Theta_b$ lie in $2\pi\ci\bbZ$ then $b$ belongs to $\calB_a$.
To investigate $\calB_a$, we introduce the notion of \emph{$A$-cycles}: homology cycles for which $\rho_\ast A = - A$.
They form a subgroup of the homology group $H_1(\Sigma_a,\bbZ)$, which is necessarily also free and abelian.
By~\eqref{eqn:Theta symmetries} any corresponding period of $\Theta_b$ is real.
To count the rank of the $A$-cycles, we deploy the following argument.
The real involution $\rho_\ast$ splits $H_1(\Sigma_a,\bbQ)$ into two eigenspaces with eigenvalues $\pm 1$.
Since $\rho$ is orientation reversing, the eigenspaces are isotropic with respect to the intersection form: $X \cdot Y = - (\rho_\ast X) \cdot (\rho_\ast Y) = - X \cdot Y$.
As the intersection form is nondegenerate, it must be a perfect pairing of the two eigenspaces.
This shows that each has dimension $g$ over $\bbQ$.
On the other hand, take a basis $\{X_1,\dots,X_{2g}\}$ for $H_1(\Sigma_a,\bbZ)$.
Observe that $X_j - \rho_\ast X_j$ is an $A$-cycle and that $X \mapsto \frac{1}{2}(X - \rho_\ast X)$ is a projection of $H_1(\Sigma_a,\bbQ)$ onto the $-1$ eigenspace of $\rho_\ast$.
Therefore the cycles $\{X_1-\rho_\ast X_1,\dots,X_{2g}-\rho_\ast X_{2g}\}$ span the $A$-cycles over $\bbQ$.
This shows that the rank of the $A$-cycles is also $g$.

Choose a basis $A_1,\dots,A_g$ of the $A$-cycles and complete it to a canonical basis of homology $A_1,\dots,A_g,B_1,\dots,B_g$.
Unless otherwise specified, we only consider bases of homology of this form.
We see $\rho_\ast B - B$ is an $A$-cycle.
As an aside, one might hope that it vanishes, but such canonical bases in general do not exist~\cite[Prop.~4.4]{Gross1981}.
The following equation implies that the real parts of the $B$-periods are given by half-integer linear combinations of the $A$-periods:
\[
\Real\left( \int_B \Theta_b \right) 
= \frac{1}{2} \left(\int_B \Theta_b - \int_{B} \rho^\ast\Theta_b\right) 
= \frac{1}{2} \left(\int_B \Theta_b - \int_{\rho_\ast B} \Theta_b\right) 
= \frac{1}{2} \int_{B - \rho_\ast B} \Theta_b.
\]
Hence if the $A$-periods of $\Theta_b$ are zero, then it has purely imaginary $B$-periods.
Indeed, $b\in\calB_a$ if and only if its $A$-periods are zero.

\begin{lemma}
\label{lem:differential omega}
The map $\calB_a \to \bbC$ given by $b \mapsto b(0)$ is an isomorphism of $\bbR$-vector spaces.  
This isomorphism endows $\calB_a$ with a complex structure $J$.
\end{lemma}
\begin{proof}
If $b(0)=0$ then $\Theta_b$ is holomorphic and Riemann's bilinear relations imply $b=0$, establishing injectivity.
On the other hand, consider $\Theta_b$ for $b = \omega + \bar{\omega}\lambda^{g+1} \in \calP^{g+1}_\bbR$.
By~\eqref{eqn:Theta symmetries}, its $A$-periods are real.
We can find a holomorphic differential $\Theta_{\tilde{b}}$ with the same $A$-periods for some $\tilde{b}$ with $\deg \tilde{b} \leq g+1$ and $\tilde{b}(0) = 0$.
Replacing $\tilde{b}$ with $\frac{1}{2}(\tilde{b} + \rho^\ast \tilde{b})$ ensures $\tilde{b} \in \calP^{g+1}_\bbR$ without changing its $A$-periods, since the $A$-periods are known to be real.
Then $b - \tilde{b} \in \calB_a$ is the sought-after polynomial with constant coefficient $\omega$.
Hence the map $b\mapsto b(0)$ is an isomorphism.
\end{proof}

This lemma shows that $\calB_a$ is a real two-dimensional subspace of $\calP^{g+1}_\bbR$.
For any $\omega\in\bbC$ let $b^\omega$ denote the unique $b^\omega\in \calB_a$ with $b^\omega(0)=\omega$.
It follows that there exists a unique basis $(b_1,b_2)$ of $\calB_a$ with $b_1(0)=1$ and $b_2(0)=\ci$, that is $b_1=b^1$ and $b_2=b^\ci$.
Throughout, an upper index on $b$ is the parameter $\omega$ and a lower index labels an element of a basis.
The complex structure $J$ is just the complex structure of $\bbC$ transported through the isomorphism to $\calB_a$: $Jb^\omega = b^{\ci \omega}$.

In general, this is as much as one can achieve towards the construction of differentials of CMC tori.
Let us make some remarks on the remaining conditions.
The pair of differentials are $\bbR$-independent if and only if their principal parts are, hence they span $\calB_a$.
If the pair have a common root then this root is common to every element of $\calB_a$ and conversely.
In this way we see that the common root condition is properly understood as a condition on $a$.
The subset of $\calH^g_0$ with a common root of $\calB_a$ on $\bbS^1$ is called $\calS^g$ and explored in~\cite{Carberry2016}.
Condition (P5), on path integrals between the two points over this common root, can be understood as a type of period condition on a singular hyperelliptic curve~\cite[Theorem~3.3]{Carberry2013a}.
A geometric interpretation of these two Sym conditions, that they correspond to translation and rotation respectively, is given in~\cite[Theorem~7.2(i)]{Kilian2007}.

Condition (P3) concerns the periods of the differentials in $\calB_a$, and how these periods vary with $a$ is the subject of this paper.
To compare periods on different curves we must compare their homology, which we do locally by means of a special type of neighborhood in $\calH^g_0$, called a \emph{distinct root neighborhood}.
Let $a \in \calH^g_0$.
For each root $\alpha$ of $a$, take a ball $B(\alpha,\varepsilon)$ with $\varepsilon$ small enough that the balls are pairwise disjoint.
Then the set of polynomials in $\calH^g_0$ such that each ball contains exactly one root is a neighborhood $U$ of $a$ in $\calH^g_0$.
The sets $\Sigma_{\tilde{a}} \setminus \lambda^{-1}[\cup_\alpha B(\alpha,\varepsilon)]$ are naturally biholomorphic for all $\tilde{a} \in U$, via projection to the $\lambda$-plane.
Since any class in homology can be represented by a cycle that avoids these balls, given a basis of homology on $\Sigma_a$, we also have a unique choice of basis on every $\Sigma_{\tilde{a}}$.
In other words, the bases of homology form a discrete fiber bundle over $\calH^g_0$.
All our results are local in $a$, so we do not work with this bundle directly.
Instead, whenever a quantity depends on a choice of basis of homology, such as the period map below, we fix a basis at the point in question and transport it to a neighborhood in this way.

\begin{definition}
\label{def:period map}
We define the \emph{period map} $\Per : \calH^g_0 \times \calP_\bbR^{g+1} \to \bbR^g \times \bbC^{g}$ as the vector of integrals $\Per_a(b) = (\int_{A_1}\Theta_b, \dots, \int_{B_g}\Theta_b)$ over a canonical basis of homology.
It is well-defined locally, as described above.
For convenience we also introduce notation for the two halves of this vector: $\APer_a(b) = (\int_{A_1}\Theta_b, \dots, \int_{A_g}\Theta_b) \in \bbR^g$ and $\BPer_a(b) = -\ci (\int_{B_1}\Theta_b, \dots, \int_{B_g}\Theta_b) \in \bbC^g$.
The plane $\calB_a$ is the kernel of $\APer_a$, and we write $\calB = \{(a,b) \in \calH^g_0 \times \calP^{g+1}_\bbR \mid b \in \calB_a\}$.
The normalization constant $-\ci$ in the definition of $\BPer_a$ is essentially taking the imaginary part, and ensures that $\BPer_a|_{\calB_a}$ takes values in $\bbR^g$.
\end{definition}

\begin{lemma}
\label{lem:differential period}
$\Per_a$ is a smooth function of $a$ and linear in $b$.
The map
\[
\calH^g_0 \times \calP^{g+1}_\bbR \to \calH^g_0 \times \bbC \times \bbR^g,
\qquad
(a,b) \mapsto (a,b(0),\APer_a(b)),
\]
is a local diffeomorphism and $\calB$ is a trivial rank $2$ vector bundle over $\calH^g_0$.
Further, $\BPer_a|_{\calB_a} : \calB_a \to \bbR^g$ has full rank and in particular is injective for $g \geq 2$.
\end{lemma}
\begin{proof}
The linearity in $b$ is just linearity of the integral, and the linear dependence of $\Theta_b$ on $b$ in its definition~\eqref{def:theta}.
As for the smoothness in $a$, we use a distinct root neighborhood in $\calH^g_0$ as described above.
For a choice of representative curves of the homology basis that lie outside of the balls, the roots of $a$ do not lie on these curves.
Therefore on the cycles the function $\nu^{-1}$ is a nonsingular function that is smooth in $a$.
An application of the dominated convergence theorem gives smoothness of the integrals.

The second claim follows from Lemma~\ref{lem:differential omega}.
If $(\dot{a},\dot{b})$ is in the kernel of the derivative of the map, then $\dot{a}=0$, and so $\dot{b}(0)=0$ and $\APer_a(\dot{b})=0$ by linearity.
Lemma~\ref{lem:differential omega} shows that $b$ is uniquely determined by $b(0)$ and its $A$-periods, hence $\dot{b}=0$.
Since both spaces have dimension $3g+1$, the map is a local diffeomorphism.
The set $\calB$ is the preimage of $\calH^g_0\times\bbC\times\{0\}$, so the local diffeomorphism restricts to the map $\calB \to \calH^g_0\times\bbC$, $(a,b)\mapsto(a,b(0))$, which is linear on fibers.
It is bijective by Lemma~\ref{lem:differential omega}, and hence a trivialization of $\calB$ as a rank $2$ vector bundle.

Next we address the claim about the rank of $\BPer_a|_{\calB_a}$.
We compute the dimension of the kernel in several cases.
Since $\calB_a$ is the kernel of $\APer_a$, the kernel of $\BPer_a|_{\calB_a}$ is the kernel of $\Per_a$.
For any $b$ in the kernel, the differential $\Theta_b$ is exact, and so gives a meromorphic function on the spectral curve with simple poles over $\lambda = 0,\infty$ with a reality condition.
For $g \geq 2$, there is no such function, since $\lambda$ is the only degree $2$ function on the spectral curve due to the uniqueness of the hyperelliptic involution.
Hence the kernel is trivial and the map is injective.

Since the claim is trivial for $g=0$, it remains to show that the kernel of the map is $1$-dimensional for $g=1$.
The meromorphic function $\lambda^{-1}\nu$ fulfills the properties, so the kernel is at least 1-dimensional.
If there is another such meromorphic function, $\bbR$-independent, then by a suitable linear combination we can construct a meromorphic function with a single simple pole.
But this is then an isomorphism to the Riemann sphere, which is a contradiction.
\end{proof}

\section{Isoperiodic Sets in \texorpdfstring{$\calH^g_0$}{Hg\_0}}
\label{sec:isoperiodic sets}

For $g \geq 2$, we know that $\BPer_a[\calB_a]$ is a $2$-dimensional linear subspace of $\bbR^g$, which we can interpret as a point in the Grassmannian $\Gr(2,\bbR^g)$.
This gives us a map from $\calH^g_0$ to $\Gr(2,\bbR^g)$.
Moreover, from Lemma~\ref{lem:differential omega}, we know that $\calB_a$ has a complex structure $J$, and for $g \geq 2$ the isomorphism of Lemma~\ref{lem:differential period} transfers this complex structure to $\BPer_a[\calB_a] \subset \bbR^g$.
Below we will show how to understand a real plane with a complex structure as a point in $\CP^{g-1}$.
In this section we investigate these two maps.
The study of Riemann surfaces through the periods of holomorphic or meromorphic differentials is a classical but difficult problem, since the relationship between the branch points and periods is nonelementary.
The main result of this section is Theorem~\ref{thm:Whitham deformation}, which shows that the level sets of these two maps in $\calH^g_0$ are submanifolds of dimensions $3$ and $1$, respectively.
Thus $\calH^g_0$ is foliated by the $3$-dimensional isoperiodic sets (away from the higher-order common roots of $\calB_a$), each of which in turn is foliated by the $1$-dimensional sets that preserve both periods and complex structure.
Theorem~\ref{thm:Willmore derivative} then describes the behavior along the $1$-dimensional level sets of a function $\willmore$ that extends the Willmore energy.

Now we explain how to incorporate the information of the complex structure.
The basic example of a real plane with a complex structure is a complex line in $\bbC^g$, where the complex structure is the restriction of multiplication by $\ci$.
Given a real plane $V$ in $\bbR^g$ with complex structure $J$, define
\[
V^\bbC
:= \{ v - \ci J v \in \bbC^g \mid v \in V \}.
\]
There is a $\bbC$-isomorphism between $(V,J)$ and $(V^\bbC,\ci)$, since
$(Jv) - \ci J(Jv) = Jv + \ci v = \ci(v - \ci Jv)$.
In particular, $V^\bbC$ is a complex line, a point in $\Gr(1,\bbC^g) = \CP^{g-1}$.
For the converse, note that $V^\bbC$ contains no purely real vectors.
Given any complex line $W \in \CP^{g-1} \setminus \RP^{g-1}$ choose a nonzero vector $w$.
Then let $W^\bbR \subset \bbR^g$ be the real span of $\Real w$ and $\Imag w$.
It is a plane, else there would be a purely real vector in $W$.
And clearly $W^\bbR$ inherits a complex structure defined on the basis by $J(\Real w) = - \Imag w$ and $J(\Imag w) = \Real w$.
These two operations are inverse to one another, see also~\cite[Section~1.2]{Huybrechts2005}, where $V^\bbC$ is denoted $V^{1,0}$.
Thus we can also understand $\BPer_a[\calB_a]^\bbC$ as a point of $\CP^{g-1} \setminus \RP^{g-1}$.

\begin{definition}
\label{def:isoperiodic sets}
Let $a\in\calH^g_0$.
The \emph{isoperiodic set} $\calH^g_0(a)$ is the connected component containing $a$ of the level set of $\tilde{a} \mapsto \BPer_{\tilde{a}}[\calB_{\tilde{a}}]$ on $\calH^g_0$.
Likewise, $\calH^g_0(a,\tau)$ is the connected component containing $a$ of the level set of $\tilde{a} \mapsto \BPer_{\tilde{a}}[\calB_{\tilde{a}}]^\bbC$ on $\calH^g_0$.
The label $\tau$ is justified by the proof of Theorem~\ref{thm:Whitham deformation}, where fixing this complex line is shown to be equivalent to fixing the ratio $\tau=b_2(0)/b_1(0)$ for a basis of $\calB_a$; its role as the conformal type of a lattice is explained in Section~\ref{sec:adapted bases}.
\end{definition}

Because of the dependence of $\BPer_a$ on a choice of homology basis, these maps are only locally defined, but that is sufficient to define level sets.

Our aim is to understand the isoperiodic sets.
Although the description in terms of Grassmannians is concise and only depends on $a$, it is easier to first work with a closely related notion of isoperiodic set that is more amenable to calculation.
For $a \in \calH^g_0$, let $b_1,b_2$ be any linearly independent elements of $\calP^{g+1}_\bbR$.
Consider the set of $(\tilde{a},\tilde{b}_1,\tilde{b}_2)\in\calH^g_0\times(\calP^{g+1}_\bbR)^2$ such that the periods of $\Theta_{\tilde{b}_k}$ are equal to those of $\Theta_{b_k}$ for $k=1,2$, which we call the \emph{isoperiodic set} of the triple $(a,b_1,b_2)$.
It is a level set of the map
\begin{align}
\label{eq:period map}
\Phi:\calH^g_0\times(\calP^{g+1}_{\bbR})^2&\to\Hom(H_1(\Sigma_a,\bbZ),\bbC^2),&
(a,b_1,b_2)&\mapsto\left(C\mapsto\int_C \Theta_{b_1},\int_C \Theta_{b_2}\right).
\end{align}
The relation to Definition~\ref{def:isoperiodic sets} is as follows: if $(b_1,b_2)$ is a basis for $\calB_a$ then $\tilde{a} \in \calH^g_0(a)$ if and only if $(\tilde{a},\tilde{b}_1,\tilde{b}_2)$ belongs to the isoperiodic set of $(a,b_1,b_2)$ for some $\tilde{b}_1,\tilde{b}_2 \in \calB_{\tilde{a}}$.
As preparation, we describe the kernel of the derivative of this map.
The description of this kernel in terms of polynomial equations is well established~\cite{Hauswirth2016,Kilian2015}.

\begin{lemma}
\label{lem:whitham kernel}
Let $a\in\calH^g_0$, and let $b_1,b_2\in\calP^{g+1}_\bbR$ be linearly independent.
We describe the kernel of $\dd\Phi$ at $(a,b_1,b_2)$.
\begin{enumerate}
\item
A tangent vector $(\dot{a},\dot{b}_1,\dot{b}_2)$ belongs to the kernel if and only if there are $c_1,c_2\in\calP^{g+1}_\bbR$ that solve the \emph{Whitham equations}:
\begin{align}
\label{eq:whitham ab}
2a\dot{b}_k - \dot{a}b_k = \ci(2\lambda a c_k' - (\lambda a' + a)c_k).
\end{align}
Then there exists $Q\in\calP^2_\bbR$ such that the \emph{compatibility equation} holds:
\begin{align}
\label{eq:whitham Q}
b_2c_1 - b_1c_2 = aQ.
\end{align}
The polynomials $c_1,c_2$ and $Q$ are uniquely determined by $(\dot{a},\dot{b}_1,\dot{b}_2)$.

\item
If $g\geq1$ and $b_1$ and $b_2$ have no common root, then $(\dot{a},\dot{b}_1,\dot{b}_2)\mapsto Q$ is an isomorphism from the kernel onto $\calP^2_\bbR$.
In particular the kernel is three-dimensional.

\item
If $g\geq1$ and $b_1$ and $b_2$ have exactly one common root $\beta$, which is of first order, then the kernel is three-dimensional and $Q$ maps it onto $\{Q\in\calP^2_\bbR\mid Q(\beta)=0\}$.
The vectors with $Q=0$ form a line, spanned by the vector with $c_k=\ci\frac{\lambda+\beta}{\lambda-\beta}b_k+\gamma b_k$, where $\gamma\in\bbR$ is the unique value for which $\dot{a}(0)=0$.

\item
If $g=0$, then $a\equiv 1$, $(b_1,b_2)$ is a basis of $\calP^1_\bbR=\calB_a$, and $b_1$ and $b_2$ have no common root.
The kernel is $\{0\}\times(\calP^1_\bbR)^2$, which is four-dimensional, and $Q$ maps it onto $\calP^2_\bbR$ surjectively.
The vectors with $Q=0$ form the line given by $c_k=\gamma b_k$ with $\gamma\in\bbR$, which give the infinitesimal rotation of the basis within $\calB_a$.
\end{enumerate}
\end{lemma}
\begin{proof}
To describe the kernel of $\dd\Phi$ we introduce the Whitham deformations of the triples $(a,b_1,b_2)$.
If the periods are fixed then for all cycles $C$ we have $\int_C \dot{\Theta}_{b_k} = 0$.
This fact, together with the fact that $\Theta_{b_k}$ is residue free, tells us that $\dot{\Theta}_{b_k}$ is an exact differential.
More precisely, $\dot{\Theta}_{b_k}$ is the exterior derivative of a meromorphic function which has first-order poles at $\lambda = 0, \infty$, is antisymmetric with respect to $\sigma$ and on which $\rho$ acts as negative complex conjugation (see~\cite[Lemma~3.1]{Carberry2019}).
This is equivalent to $\dot{\Theta}_{b_k} = \dd \left( \frac{\ci c_k(\lambda)}{\nu} \right)$ for some $c_k \in \calP^{g+1}_\bbR$.
Inserting \eqref{def:theta} for $\Theta_{b_k}$, differentiating in $t$, and clearing denominators gives the Whitham equation~\eqref{eq:whitham ab}.
Conversely, if~\eqref{eq:whitham ab} holds then $\dot{\Theta}_{b_k} = \dd \left( \frac{\ci c_k}{\nu} \right)$ is exact, so all its periods vanish.

To derive the compatibility equation, we eliminate $\dot{a}$ by multiplying these equations for $k=1,2$ with $b_2,b_1$, respectively, and subtracting.
The result is
\begin{equation}
\label{eq:whitham antisymmetrize}
2a(b_2\dot{b}_1 - b_1\dot{b}_2)
= \ci \left[ 2\lambda a (b_2 c_1' - b_1 c_2') - (\lambda a' + a)(b_2 c_1 - b_1 c_2) \right].
\end{equation}
Every term contains either $a$ or $\lambda a'$ and rearranging shows that $a$ divides $\ci \lambda a' (b_2 c_1 - b_1 c_2)$.
As every root of $a$ is simple, $a$ is coprime to $a'$.
Moreover, since $a(0) = 1$, no root of $a$ is zero.
Hence $a$ divides $b_2 c_1 - b_1 c_2$ and the compatibility equation~\eqref{eq:whitham Q} holds for some polynomial $Q \in \calP^2_{\bbR}$.

We note the following general argument that for a given $(\dot{a},\dot{b}_1,\dot{b}_2)$ the equations~\eqref{eq:whitham ab} and~\eqref{eq:whitham Q} determine polynomials $c_k$ and $Q$ uniquely, if they exist.
The linear map 
\begin{equation}
\label{eq:c injective}
c\mapsto 2\lambda ac'-(\lambda a'+a)c = 2\nu^3 \frac{\dd}{\dd\lambda}\left( \frac{c}{\nu} \right)
\end{equation}
is injective on $\calP^{g+1}_\bbR$, since for any $c$ in the kernel the function $c/\nu$ is constant, which implies $c=0$.
Moreover, wherever $c_k$ exists it depends smoothly on $a \in \calH^g_0$ and on $(\dot{a},\dot{b}_k)$, as the solution of a linear system in the coefficients.
Likewise, since $aQ$ is polynomial in $b_1,b_2,c_1,c_2$ and multiplication by $a$ is injective, so too $Q$ is unique and smooth in the data.
This is used in the proof of Theorem~\ref{thm:Willmore derivative}.
This completes the proof of~(i).

Two solutions of~\eqref{eq:whitham ab} will be useful.
The triple $(2a,b_1,b_2)$ solves it with $c_k=0$, and the \emph{trivial solution}
\begin{align*}
(\dot{a},\dot{b}_1,\dot{b}_2)
=\ci\Big(\lambda a'-ga,\lambda b_1'-\tfrac{g+1}{2}b_1,\lambda b_2'-\tfrac{g+1}{2}b_2\Big)
\end{align*}
solves it with $c_k=b_k$.
The trivial solution belongs to $\calP^{2g}_\bbR\times(\calP^{g+1}_\bbR)^2$ by the following computation for $p\in\bbC^m[\lambda]$:
\begin{align*}
m\lambda^m\overline{p(\bar{\lambda}^{-1})}-2\lambda\tfrac{\dd}{\dd\lambda}\left(\lambda^m\overline{p(\bar{\lambda}^{-1})}\right)&=m\lambda^m\overline{p(\bar{\lambda}^{-1})}-2m\lambda^m\overline{p(\bar{\lambda}^{-1})}+2\lambda^{m+1}\overline{\bar{\lambda}^{-2}\tfrac{\dd}{\dd\bar{\lambda}^{-1}}p(\bar{\lambda}^{-1})}\\&=-\lambda^m\left(\overline{mp(\bar{\lambda}^{-1})-2\bar{\lambda}^{-1}p'(\bar{\lambda}^{-1})}\right).
\labelthis{eq:derivative conjugate}
\end{align*}
For $g\geq1$ these two solutions have $\dot{a}(0)=2$ and $\dot{a}(0)=-\ci g$, respectively, so adding real multiples of them to any solution of~\eqref{eq:whitham ab} achieves $\dot{a}(0)=0$.

Now let $g\geq1$.
For (ii) and (iii) we need a lower bound on the dimension of the kernel.
The $A$-periods of $\Theta_{b_k}$ are real, and the real parts of its $B$-periods are determined by its $A$-periods, by the formula preceding Lemma~\ref{lem:differential omega}.
Hence $\Phi$ takes values in a real vector space of dimension $4g$.
Since $\calH^g_0$ is open in $\{a\in\calP^{2g}_\bbR \mid a(0)=1\}$, which has dimension $2g-1$, the dimension of $\ker\dd\Phi$ is at least $(2g-1)+2(g+2)-4g = 3$.

For (ii) we show that the map from the kernel to $Q\in\calP^2_\bbR$ is injective.
It then follows that the dimension of the kernel is at most $\dim\calP^2_\bbR = 3$, and therefore exactly $3$.
Assume therefore that $(\dot{a},\dot{b}_1,\dot{b}_2)\in\ker\dd\Phi$ and that the corresponding $Q=0$.
Then $b_2c_1 = b_1c_2$, and since $b_1$ and $b_2$ have no common root, $(c_1,c_2) = r(b_1,b_2)$ for some $r\in\bbR$.
Subtracting $r$ times the trivial solution leaves a solution of~\eqref{eq:whitham ab} with $c_k=0$, that is, $2a\dot{b}_k = \dot{a}b_k$ for $k=1,2$.
Every root of $a$ is simple and is not a common root of $b_1$ and $b_2$, so $a$ divides the first component of this solution.
This component and $a$ both lie in $\calP^{2g}_\bbR$, so it equals $2sa$ for some $s\in\bbR$, and then the other components are $sb_k$.
Hence (compare~\cite[Lemma~4.6]{Carberry2019})
\[
(\dot{a},\dot{b}_1,\dot{b}_2) = r\ci\Big(\lambda a'-ga,\lambda b_1'-\tfrac{g+1}{2}b_1,\lambda b_2'-\tfrac{g+1}{2}b_2\Big) + s(2a,b_1,b_2).
\]
Since $a(0)=1$, evaluating the first component at $\lambda=0$ gives $0 = \dot{a}(0) = -\ci g r + 2s$.
As $g\geq1$, both $r,s\in\bbR$ vanish, and therefore $(\dot{a},\dot{b}_1,\dot{b}_2)=0$.
This proves~(ii).

For (iii), suppose that $b_1$ and $b_2$ have exactly one common root $\beta$, which is of first order.
The roots of a polynomial in $\calP^{g+1}_\bbR$ are symmetric under $\lambda\mapsto\bar{\lambda}^{-1}$, with $0$ paired with $\infty$, so a single common root satisfies $|\beta|=1$.
Write $b_k = p_\beta\tilde{b}_k$ with $p_\beta = \ci\beta^{-1/2}(\lambda-\beta)$.
Then $\tilde{b}_1$ and $\tilde{b}_2$ have no common root, and we assume $\tilde{b}_1(\beta)\neq0$ without loss of generality.
Evaluating~\eqref{eq:whitham Q} at $\beta$ gives $a(\beta)Q(\beta)=0$, so $Q(\beta)=0$.
By the reality condition, $\beta^{-1}Q(\beta)$ is real, so this is a single real condition and $\{Q\in\calP^2_\bbR \mid Q(\beta)=0\}$ is two-dimensional.
We show that the map
\[
(\dot{a},\dot{b}_1,\dot{b}_2) \mapsto \big(\beta^{-(g+1)/2}c_1(\beta), Q\big)
\in \bbR \times \{Q\in\calP^2_\bbR \mid Q(\beta)=0\}
\]
is injective on the kernel, where $\beta^{-(g+1)/2}c_1(\beta)$ is real by the reality condition $c_1 \in \calP^{g+1}_\bbR$ for any fixed choice of the square root.
Suppose therefore that $c_1(\beta) = 0$ and $Q = 0$.
Then $\tilde{b}_2c_1 = \tilde{b}_1c_2$, so $(c_1,c_2) = v(\tilde{b}_1,\tilde{b}_2)$ for some $v\in\calP^1_\bbR$.
Moreover $v(\beta)\tilde{b}_1(\beta) = c_1(\beta) = 0$ gives $v(\beta)=0$, so $v$ is a real multiple of $p_\beta$ and $(c_1,c_2) = r(b_1,b_2)$ for some $r\in\bbR$.
Since $|\beta|=1$ and $a$ has no roots on $\bbS^1$, $\beta$ is not a root of $a$, and the rest of the proof of~(ii) applies verbatim.
As the kernel has dimension at least $3$, this injective map is an isomorphism onto the $3$-dimensional space $\bbR \times \{Q\in\calP^2_\bbR \mid Q(\beta)=0\}$.
In particular $Q$ maps the kernel onto $\{Q\in\calP^2_\bbR \mid Q(\beta)=0\}$, and the vectors with $Q=0$ form a line on which $c_1(\beta)\neq0$.
On this line $(c_1,c_2) = v(\tilde{b}_1,\tilde{b}_2)$ with $v\in\calP^1_\bbR$ and $v(\beta)\neq0$.
Since $|\beta|=1$, the space $\calP^1_\bbR$ is spanned by $p_\beta$ and $\beta^{-1/2}(\lambda+\beta)$, and $\beta^{-1/2}(\lambda+\beta)\tilde{b}_k = -\ci\frac{\lambda+\beta}{\lambda-\beta}b_k$.
So after rescaling, the line is spanned by the vector with $c_k=\ci\frac{\lambda+\beta}{\lambda-\beta}b_k+\gamma b_k$ for some $\gamma\in\bbR$, which is unique because the line is one-dimensional.
This proves~(iii).

Finally let $g=0$.
There is only one spectral curve in this case, $a \equiv 1$, and $\calH^0 = \{1\}$ is zero-dimensional.
There are no cycles, so $\Phi$ is trivial and the kernel of $\dd\Phi$ is the whole tangent space $\{0\}\times(\calP^1_\bbR)^2$.
Since $\calB_a\subset\calP^1_\bbR$ is two-dimensional, $\calB_a=\calP^1_\bbR = \{ r\ci \beta^{-1/2}(\lambda - \beta) \mid r \in \bbR, \beta \in \bbS^1\}$.
The linearly independent $(b_1,b_2)$ are a basis of it and have no common root.
With $a=1$ and $\dot{a}=0$ the Whitham equations read $2\dot{b}_k=\ci(2\lambda c_k'-c_k)$, which for every $(c_1,c_2)\in(\calP^1_\bbR)^2$ determines a unique kernel vector, and $Q=b_2c_1 - b_1c_2$.
Unlike the case $g\geq1$, the trivial solution $(0,b_1,b_2)$ is not removed by the condition $a(0)=1$.
If $Q=0$ it follows that $c_k = rb_k$ for some $r \in \bbR$, since $b_1,b_2$ are coprime.
Therefore the map $(c_1,c_2)\mapsto Q$, from the four-dimensional space $(\calP^1_\bbR)^2$ to the three-dimensional space $\calP^2_\bbR$, has one-dimensional kernel and is surjective.
The kernel vector for $c_k=rb_k$ has $\dot{b}_k(0)=-\frac{\ci r}{2}b_k(0)$ and is thus an infinitesimal rotation of $\calB_a$.
\end{proof}

With those preparations made, we can now show that the isoperiodic sets are submanifolds under mild assumptions.

\begin{theorem}
\label{thm:Whitham deformation}
Let $g\geq2$, let $a\in \calH^g_0$, and suppose that the elements of $\calB_a$ have at most one common root, and that it is of first order.
The level set $\calH^g_0(a)$ of $\tilde{a} \mapsto \BPer_{\tilde{a}}[\calB_{\tilde{a}}] \in \Gr(2,\bbR^g)$ locally at $a$ is a three-dimensional submanifold of $\calH^g_0$.
Further, the level set $\calH^g_0(a,\tau)$ of $\tilde{a} \mapsto \BPer_{\tilde{a}}[\calB_{\tilde{a}}]^\bbC \in \CP^{g-1} \setminus \RP^{g-1}$ is locally at $a$ a one-dimensional submanifold of $\calH^g_0$.
\end{theorem}

A curve with a single differential has a universal deformation space~\cite{Hauswirth2017arxiv,Carberry2025arxiv};
common roots of a pair of differentials are an obstacle to a smooth deformation theory.

\begin{proof}
Choose a basis $(b_1,b_2)$ of $\calB_a$.
First note that the isoperiodic set of the triple $(a,b_1,b_2)$ is a submanifold of dimension $3$, because it is the level set of $\Phi$ from~\eqref{eq:period map} and Lemma~\ref{lem:whitham kernel} showed that its rank is constant.
Let $\pi$ be the projection from the isoperiodic set to $\calH^g_0$ given by $(a,b_1,b_2) \mapsto a$.
We claim that $\pi$ is a proper injective immersion.
It then follows by~\cite[Prop.~4.22]{Lee2012} that the image of the isoperiodic set of triples is a closed submanifold of $\calH^g_0$.

For injectivity, suppose both $(a,b_1,b_2)$ and $(a,\tilde{b}_1,\tilde{b}_2)$ belong to the isoperiodic set.
But then $\BPer_a(b_k - \tilde{b}_k) = 0$ and Lemma~\ref{lem:differential period} gives $b_k = \tilde{b}_k$.
Next, the kernel of $\dd\pi$ consists of vectors of the form $(0,\dot b_1,\dot b_2)$.
Since $\dot{a}=0$, the derivative $\frac{\dd}{\dd t}\Theta_{b_k}=\Theta_{\dot{b}_k}$ has all periods zero, so $\dot{b}_k=0$ by Lemma~\ref{lem:differential period}.
Thus $\pi$ is an immersion.

To show properness, we first bound the differentials by constructing local inverses to the period map.
In order to speak of boundedness, we equip the finite dimensional space $\calP^{g+1}_\bbR$ with some norm.
Fix an $a \in \calH^g_0$ and choose a complementary subspace of $\BPer_{a}[\calB_{a}] \subset \bbR^g$.
This defines a projection $\mathrm{pr}_a : \bbR^g \to \BPer_{a}[\calB_{a}]$.
The composition $\mathrm{pr}_a \circ \BPer$ then is a smooth bundle morphism from $\calB$ to the trivial bundle $\calH^g_0\times \BPer_{a}[\calB_{a}]$.
Its derivative at $a$ is the identity by construction, so there is a neighborhood of $a$ on which there is a smooth inverse.
Now take any compact subset $K \subset \calH^g_0$.
We show sequential compactness of $\pi^{-1}[K]$.
For any sequence in this set, we first pass to a subsequence such that the image under $\pi$ converges in $K$.
There is an inverse to the period map in a neighborhood of this limit $a \in K$.
Since $\pi^{-1}[K]$ is contained in the isoperiodic set by definition, the sequence is the image of a constant under the inverse of the period map, and therefore converges.
This proves properness of $\pi$, and the first statement of the theorem.

Next suppose that at $a$ we choose $b_1,b_2$ such that $b_2 = J b_1$. 
We show that $b_2 = Jb_1$ continues to hold if and only if the Whitham flow fixes $\BPer_a[\calB_a]^\bbC$.
If $\BPer_a[\calB_a]^\bbC$ is fixed, so too is $\BPer_a[\calB_a]$ and we are in the isoperiodic set.
Further, it is the complex span of
\[
\BPer_a(b_1) - \ci \BPer_a(Jb_1) 
= \left(-\ci\int_{B_j} \Theta_{b_1} - \int_{B_j} \Theta_{Jb_1}\right)_{j=1,\dots,g}.
\]
Since the real part $\BPer_a(b_1)$ of this vector is fixed, so too must the imaginary part $\BPer_a(Jb_1)$ be.
At $t=0$, $Jb_1 = b_2$ since $b_2(0) = \ci$ and $b_1(0) = 1$.
Therefore $\int_{B_j} \Theta_{Jb_1} = \int_{B_j} \Theta_{b_2}$ for all $j$ along the flow.
Since a differential in $\calB_a$ is determined by its periods (Lemma~\ref{lem:differential period}), we conclude that $b_2 = Jb_1$ along any flow that fixes $\BPer_a[\calB_a]^\bbC$.
The converse follows by definition.
Therefore the flow that fixes $\BPer_a[\calB_a]^\bbC$ amounts to the additional constraint $b_2(0) = \ci b_1(0)$ along the flow, for this initial choice with $b_2 = J b_1$.
Were we to choose an arbitrary basis $(b_1,b_2)$ initially, with $b_2(0) = \tau b_1(0)$ for some $\tau \in \bbC$, then this equation holds along the flow if and only if $\BPer_a[\calB_a]^\bbC$ is fixed.

The map $(a,b_1,b_2)\mapsto\tau=b_2(0)/b_1(0)$ is a submersion on the isoperiodic set, as we now show.
From the formula~\eqref{eq:whitham ab} we get
\begin{equation}
\label{eq:dot tau Q}
\dot{\tau}
= \frac{\dot{b}_2(0)b_1(0) - b_2(0)\dot{b}_1(0)}{b_1(0)^2}
= \frac{\ci}{2} \frac{b_2(0)c_1(0) - c_2(0)b_1(0)}{b_1(0)^2}
= \frac{\ci a(0)}{2b_1(0)^2}Q(0).
\end{equation}
Hence the kernel of $\dd\tau$ on the tangent space of the isoperiodic set consists of the vectors with $Q(0)=0$, that is, $Q = Q'(0)\lambda$.
If $b_1$ and $b_2$ have no common root, then by Lemma~\ref{lem:whitham kernel}(ii) $Q$ ranges over $\calP^2_\bbR$, so $Q(0)$ ranges over $\bbC$, and the kernel of $\dd\tau$ is the line $\{Q\in\bbR\lambda\}$.
If they have a common root $\beta$, then by Lemma~\ref{lem:whitham kernel}(iii) $Q$ ranges over $\{Q\in\calP^2_\bbR\mid Q(\beta)=0\}$.
If moreover $Q(0)=0$, then $Q=Q'(0)\lambda$ with $Q'(0)\beta=0$, so $Q=0$.
Hence $Q\mapsto Q(0)$ maps $\{Q\in\calP^2_\bbR\mid Q(\beta)=0\}$ isomorphically onto $\bbC$, and the kernel of $\dd\tau$ is the line of vectors with $Q=0$.
In both cases $\dd\tau$ is surjective, and its kernel is one-dimensional.
This proves the second statement of the theorem.
\end{proof}

For the $1$-dimensional submanifold in particular, it is natural to look for a real-valued function that can serve as a parameter.
In our study of the genus-$2$ case, which is always doubly periodic, we observed that the Willmore energy $\int H^2\,\dd A$ plays this role.
There exists a formula for the Willmore energy of a CMC immersion with doubly periodic metric~\cite[Theorem~5.7]{Knopf2025}.
Suppose $(b_1,b_2)$ is a basis of the lattice of differentials that meet all conditions (P1)--(P3) of Section~\ref{sec:differentials}, so that they determine a fundamental domain of the metric.
The Willmore energy of this piece of the immersion is
\begin{equation}
\label{eq:willmore residue}
2\ci\Res_{\lambda=0}q_{b_2}\,\Theta_{b_1},
\end{equation}
where $q_{b_2}$ is a local primitive of $\Theta_{b_2}$ near $\lambda=0$.
The residue does not depend on the choice of primitive, since $\Theta_{b_1}$ has no residue.
The expression~\eqref{eq:willmore residue} is real for any $b_1,b_2\in\calB_a$, as we see by comparing the residues at $\lambda=0$ and $\lambda=\infty$.
Cutting the curve along a canonical basis of homology gives a simply connected polygon on which $q_{b_2}$ is single-valued, and the reciprocity law~\cite[p.~241]{Griffiths1994} gives
\[
2\pi\ci\left(\Res_{\lambda=0} + \Res_{\lambda=\infty}\right) q_{b_2}\,\Theta_{b_1}
= \sum_j \left( \int_{A_j}\Theta_{b_2} \int_{B_j}\Theta_{b_1} - \int_{B_j}\Theta_{b_2} \int_{A_j}\Theta_{b_1} \right)
= 0,
\]
since the $A$-periods of differentials of $\calB_a$ vanish.
On the other hand, $\rho^\ast\Theta_b = -\overline{\Theta_b}$ implies $\rho^\ast q_b = -\overline{q_b} + C$ for a constant $C$, so $\overline{\rho^\ast(q_{b_2}\Theta_{b_1})} = q_{b_2}\Theta_{b_1} - \bar{C}\,\Theta_{b_1}$.
As $\rho$ is antiholomorphic and exchanges $\lambda = 0,\infty$, and $\Theta_{b_1}$ has no residues, we obtain $\Res_{\lambda=\infty} q_{b_2}\Theta_{b_1} = \overline{\Res_{\lambda=0} q_{b_2}\Theta_{b_1}}$.
Together these show that $\Res_{\lambda=0} q_{b_2}\Theta_{b_1}$ is imaginary.

The residue in~\eqref{eq:willmore residue} can be computed explicitly in terms of $b_1$ and $b_2$.
The point of $\Sigma_a$ over $\lambda=0$ is a branch point, so we use the local coordinate $\zeta$ with $\lambda = \zeta^2$ and $\nu = \zeta\sqrt{a(\zeta^2)}$, where the square root is the branch with value $1$ at $\zeta=0$.
Recall that $a(0)=1$, so that $1/\sqrt{a(\lambda)} = 1 - \tfrac{1}{2}a'(0)\lambda + O(\lambda^2)$.
Hence for any $b\in\calP^{g+1}_\bbR$
\[
\Theta_b
= \frac{b(\zeta^2)}{\sqrt{\zeta^2 a(\zeta^2)}}\,\frac{2\zeta\,\dd \zeta}{\zeta^2}
= 2\left(b(0)\zeta^{-2} + b'(0) - \tfrac{1}{2}a'(0)b(0) + O(\zeta^2)\right)\dd \zeta,
\]
and $q_b = -2b(0)\zeta^{-1} + C + 2\bigl(b'(0) - \tfrac{1}{2}a'(0)b(0)\bigr)\zeta + O(\zeta^3)$ for a constant $C$.
The residue of $q_{b_2}\Theta_{b_1}$ is the coefficient of $\zeta^{-1}\dd \zeta$ in the product:
\begin{align*}
\Res_{\lambda=0}q_{b_2}\,\Theta_{b_1}
&= 4\bigl(b_1(0)b_2'(0) - b_1'(0)b_2(0)\bigr).
\end{align*}
The other choice of sign of $\nu$ changes the sign of both $\Theta_{b_1}$ and $q_{b_2}$, and so does not change the residue.
This should be compared with the similar energy formula of Hitchin~\cite[Theorem~13.17(ii)]{Hitchin1990} for conformal harmonic maps from a torus to $\bbS^3$.

This expression for the residue shows why it is independent of the choice of differentials that span the lattice: a change of basis scales the residue by the determinant, and a change of basis of a lattice has determinant $1$.
Therefore to extend this formula to an arbitrary basis of $\calB_a$, we divide out by the area of the parallelogram spanned by the periods of $b_1$ and $b_2$, which also scales by the determinant under change of basis.

\begin{definition}
\label{def:willmore}
Let $a\in \calH^g_0$ and choose a basis $(b_1,b_2)$ of $\calB_a$.
We define a quantity that we call the normalized Willmore energy by
\begin{align}
\label{eq:willmore}
\willmore(a) = \frac{8\ci\bigl(b_1(0)b_2'(0) - b_1'(0)b_2(0)\bigr)}{\operatorname{Area}(\BPer_a(b_1),\BPer_a(b_2))} ,
\end{align}
where $\operatorname{Area}$ is the area of the parallelogram spanned by the two vectors in $\bbR^g$, with sign given by the orientation from the complex structure $J$.
\end{definition}

Our general approach is to treat $\calH^g$ as the primary space, the set of hyperelliptic curves, rather than the space of triples $(a,b_1,b_2)$.
Therefore it is useful to interpret $\willmore$ as an object that only depends on $a$.
The above discussion shows that this is almost true: the definition of $\BPer_a$ depends on the choice of canonical basis of homology.
Changing the canonical basis changes $\bbR^g$ by a $\GL(g,\bbZ)$ matrix (see~\eqref{eq:symplectic transformation} for more details), which is not area preserving for $g \geq 3$.
If we restrict the bundle $\calB$ to a level set $\calH^g_0(a)$, the plane $\calB_{\tilde{a}}$ for any $\tilde{a} \in \calH^g_0(a)$ can be identified with the fixed plane $\BPer_{a}[\calB_{a}]$ and the issue is mitigated.
We now ask ``how does $\willmore$ change along a curve in $\calH^g_0(a)$?''

\begin{theorem}
\label{thm:Willmore derivative}
Let $a\in\calH^g_0$ be as in Theorem~\ref{thm:Whitham deformation}.
The derivative $\dot{\willmore}$ along the one-dimensional submanifold $\calH^g_0(a,\tau)$ vanishes at $a \in \calH^g_0$ if and only if $\calB_a$ has a common root.
At $a \in \calH^g_0$ with a common root $\beta$ of $\calB_a$ of first order, the second derivative $\ddot{\willmore}$ vanishes if and only if $\frac{\dd}{\dd\lambda}\frac{b_2}{b_1}$ vanishes at $\beta$.
\end{theorem}

\begin{proof}
Take the setup of Theorem~\ref{thm:Whitham deformation}.
That is, we choose any basis $(b_1,b_2)$ of $\calB_a$ at $a \in \calH^g_0$.
Then on $\calH^g_0(a)$ we have a frame $(b_1,b_2)$ of $\calB$ with constant periods.
The area $\operatorname{Area}(\BPer_a(b_1),\BPer_a(b_2))$ of the parallelogram is constant for these choices, and we abbreviate it to $\operatorname{Area}$.
Therefore any change of $\willmore$ comes from the numerator.
Since $a(0)=1$ along the curve, $\dot{a}(0)=0$.
Comparing the coefficients of $\lambda^0$ and $\lambda^1$ in~\eqref{eq:whitham ab} gives
\[
\dot{b}_k(0) = -\tfrac{\ci}{2}c_k(0),
\qquad
\dot{b}_k'(0) = \tfrac{\ci}{2}\bigl(c_k'(0) - a'(0)c_k(0)\bigr) + \tfrac{1}{2}\dot{a}'(0)b_k(0),
\]
and the same comparison in~\eqref{eq:whitham Q} gives
\begin{align*}
Q(0) &= b_2(0)c_1(0) - b_1(0)c_2(0), \\
Q'(0) + a'(0)Q(0) &= b_2'(0)c_1(0) + b_2(0)c_1'(0) - b_1'(0)c_2(0) - b_1(0)c_2'(0).
\end{align*}
We differentiate Equation~\eqref{eq:willmore} for $\willmore$ and substitute these:
\begin{equation}
\label{eq:Willmore derivative}
\begin{aligned}
\dot{\willmore} \operatorname{Area}
&= 8\ci\bigl(\dot{b}_1(0)b_2'(0) + b_1(0)\dot{b}_2'(0) - \dot{b}_1'(0)b_2(0) - b_1'(0)\dot{b}_2(0)\bigr) \\
&= 4\bigl(b_2'(0)c_1(0) + b_2(0)c_1'(0) - b_1'(0)c_2(0) - b_1(0)c_2'(0)\bigr) \\
&\qquad - 4a'(0)\bigl(b_2(0)c_1(0) - b_1(0)c_2(0)\bigr) \\
&= 4 Q'(0).
\end{aligned}
\end{equation}
Along $\calH^g_0(a,\tau)$ we have $Q = Q'(0)\lambda$, and by the proof of Theorem~\ref{thm:Whitham deformation}, which rests on Lemma~\ref{lem:whitham kernel}(ii) and~(iii), $Q'(0) = 0$ if and only if $(b_1,b_2)$ has a common root.

Suppose $(b_1,b_2)$ has a simple common root at $\beta$, and assume that $b_1$ has a simple root at $\beta$; if it does not, then $b_2$ does, and the essence of the proof is the same.
Define $b_3 = b_2 - r b_1$ for $r = \left(\frac{b_2}{b_1}\right)(\beta)$, the unique element of $b_2 + b_1 \bbC$ with a root at $\beta$ of order at least two.
Applying the reality conditions shows that in fact $r \in \bbR$.
Let $c_3 = c_2 - rc_1$.
As $Q \equiv 0$, we know $(c_1,c_2,c_3)$ lie in $(b_1,b_2,b_3)\bbR+(\ci\frac{\lambda+\beta}{\lambda-\beta}b_1,\ci\frac{\lambda+\beta}{\lambda-\beta}b_2, \ci\frac{\lambda+\beta}{\lambda-\beta}b_3)\bbR$, and for a nonzero tangent vector $(\dot{a},\dot{b}_1,\dot{b}_2)$ the second real scalar is nonzero.
Hence $c_1(\beta) \neq 0$ and the order of the root of $c_3$ at $\beta$ is one less than the order of the root of $b_3$.
In particular, $c_3$ has at least a simple root at $\beta$.
By taking the same linear combination of~\eqref{eq:whitham ab}, we see that the order of the root of $\dot{b}_3$ at $\beta$ is one less than that of $c_3$.

In~\eqref{eq:whitham Q}, substitute $(b_2,c_2) = (b_3 + rb_1, c_3 + rc_1)$, and since $Q\in\lambda\bbR$ holds on the submanifold $\calH^g_0(a,\tau)$, substitute $Q = Q'(0)\lambda$.
Differentiating along the submanifold and evaluating at $\beta$ gives
\begin{align*}
\dot{b}_3c_1 - \dot{b}_1c_3
+ b_3\dot{c}_1 - b_1\dot{c}_3
&= \dot{a}Q'(0)\lambda + a\dot{Q}'(0)\lambda, \\
\dot{b}_3(\beta)c_1(\beta) - 0
+ 0 - 0
&= 0 + a(\beta)\dot{Q}'(0)\beta.
\end{align*}
This shows that $\dot{Q}'(0) = 0$ if and only if $\dot{b}_3(\beta) = 0$, and therefore by the above consideration of orders if and only if $b_3$ has a root of order at least three.
This is equivalent to the vanishing of the following derivative at $\beta$
\[
\frac{\dd}{\dd\lambda}\frac{b_2}{b_1}
= \frac{\dd}{\dd\lambda}\left(\frac{b_3}{b_1} + r\right)
= \frac{\dd}{\dd\lambda}\frac{b_3}{b_1}.
\]
Along the submanifold the second derivative satisfies $\ddot{\willmore}\operatorname{Area} = 4\dot{Q}'(0)$, completing the proof of the theorem.
\end{proof}

Let us interpret Theorem~\ref{thm:Willmore derivative} in terms of the winding number function~\cite[Section~3]{Carberry2016}.
For an oriented basis $(b_1,b_2)$ of $\calB_a$ define
\begin{equation}
\label{eq:def f}
f=\frac{b_2}{b_1},
\qquad
\tilde{f}=\frac{b_1+\ci b_2}{b_1-\ci b_2}=\frac{1+\ci f}{1-\ci f}.
\end{equation}
Our $f$ is the reciprocal of the function $b_1/b_2$ of~\cite[Section~3]{Carberry2016}; since only its degree as a holomorphic map is used in that paper, this makes no difference.
The above convention has the advantage that $f(0) = b_2(0)/b_1(0) = \tau$.
The function $\tilde{f}$ agrees with that of~\cite{Carberry2016}.
Since $b_1,b_2\in\calP^{g+1}_\bbR$, the function $f$ is real-valued on $\bbS^1$, and so $\tilde{f}$ maps $\bbS^1$ to itself.
The degree of $\tilde{f}\colon\bbS^1\to\bbS^1$ is the \emph{winding number}.
A change of oriented basis of $\calB_a$ changes $f$ by a real Möbius transformation with positive determinant, so the winding number depends only on $a$.
Away from $\calS^g$, the subset where $\calB_a$ has a common root on $\bbS^1$, $f$ depends continuously on $a$, so the winding number is locally constant on $\calH^g_0\setminus\calS^g$.
Theorem~\ref{thm:Willmore derivative} says that $\willmore$ is strictly monotone along $\calH^g_0(a,\tau)$ away from $\calS^g$.
We hypothesize that the winding number is connected with whether $\willmore$ is increasing or decreasing.
At $a\in\calS^g$ with a first-order common root $\beta$ we have $\dot{\willmore}=0$, and $\ddot{\willmore}\neq0$ if and only if $\dd f$ does not vanish at $\beta$, in which case $a$ is a strict local extremum of $\willmore$ along $\calH^g_0(a,\tau)$.
In particular, $\calH^g_0(a,\tau)$ is transverse to $\calS^g$ at such extrema.

\section{Adapted Canonical Bases of Homology}
\label{sec:adapted bases}

Let us motivate this section with the following problem.
Consider one of the $3$-dimensional isoperiodic sets $\calH^g_0(a)$ and another $\tilde{a} \in \calH^g_0$: is there an invariant to say whether $\tilde{a} \in \calH^g_0(a)$?
A first attempt to answer this question would note that $\calH^g_0(a)$ lies in a level set of $\tilde{a} \mapsto \BPer_{\tilde{a}}[\calB_{\tilde{a}}] \in \Gr(2,\bbR^g)$, so the value of the level itself is an invariant.
The slight complication is that this value is only locally well-defined, since $\BPer_a$ requires the choice of a basis of homology.
Recall from Lemma~\ref{lem:differential period} that we use a canonical basis of homology where the $A$-cycles have $\rho_\ast A = - A$ to define this map (and ensure that its image lies in $\bbR^g$).
Any two such bases differ by a symplectic matrix of the form
\begin{equation}
\label{eq:symplectic transformation}
\begin{pmatrix}
(X^T)^{-1} & 0 \\ Z & X
\end{pmatrix}
\quad\text{with}\quad
X \in \GL(g,\bbZ) \text{ and } ZX^T = (ZX^T)^T.
\end{equation}
Explicitly, the new basis of homology has $A$-cycles $A'_i = \sum_j ((X^T)^{-1})_{ij}A_j$ and $B$-cycles $B'_i = \sum_j Z_{ij}A_j + X_{ij} B_j$, and the components of $\BPer_{a}(b) \in \bbR^g$ with respect to the new basis are, up to the factor $-\ci$,
\[
\int_{B'_i} \Theta_b
= \sum_j Z_{ij} \int_{A_j} \Theta_b + \sum_j X_{ij} \int_{B_j} \Theta_b
= \sum_j X_{ij} \int_{B_j} \Theta_b,
\]
since $\APer_{a}(b) = 0$ for $b \in \calB_a$.
In short, the change of basis induces a linear action of $X \in \GL(g,\bbZ)$ on $\bbR^g$, which in turn induces an action on $\Gr(2,\bbR^g)$.
Hence the well-defined invariant is the $\GL(g,\bbZ)$-orbit of $\BPer_a[\calB_a]$ in $\Gr(2,\bbR^g)$.

To understand the values of this invariant, we should quotient $\Gr(2,\bbR^g)$ by this action, but in general the quotient is complicated.%
\footnote{Just consider the analogous action of $\GL(2,\bbZ)$ on $\RP^1$. 
Two lines lie in the same orbit if and only if the continued fractions of their gradients eventually agree~\cite[Theorem~175]{Hardy1975}, so every orbit is countable and dense.}
Therefore we restrict now to the doubly periodic case, to which CMC tori belong.
Precisely, we consider planes $\Pi \in \Gr(2,\bbR^g)$ such that $\Lambda := \Pi \cap \bbZ^g$ is a rank two lattice.
These are precisely the elements of the subset $\Gr(2,\bbQ^g)$.
Choose generators $v_1,v_2$ of $\Lambda$ and, by inductively choosing integer vectors that are primitive in $\bbZ^g/\bbZ\langle v_1,\dots,v_k\rangle$, extend this to a generating set $v_1,\dots,v_g$ of $\bbZ^g$.
Then the matrix $(v_1|\dots|v_g) \in \GL(g,\bbZ)$ transforms the standard generating set of $\bbZ^g$ to this one, and in particular transforms $\bbR^2 \times \{0\} \subset \bbR^g$ to $\Pi$.
This demonstrates that all $a\in \calH^g_0$ in the doubly periodic case have planes $\Pi = \BPer_a[\calB_a]$ that belong to the same $\GL(g,\bbZ)$-orbit in $\Gr(2,\bbR^g)$.
Thus our first attempt to find an invariant cannot distinguish the sets $\calH^g_0(a)$ in the doubly periodic case.

We can do better!
The core idea is to be more selective about the choice of basis of homology.
A common choice in the literature~\cite{Carberry2016a} is the following: if the pairs of roots all lie on distinct rays in the $\lambda$-plane, which is the generic case, then one may lift the line segments between pairs of roots to distinct cycles in homology and complete these to a canonical basis.
This choice is not well-defined if pairs of roots pass by one another, leading to branch cut issues.
That this entanglement preserves the basis of homology modulo $2$ is classical~\cite[Lemma~IIIa.8.12]{Mumford1984}, but this observation has not been used in the spectral-curve literature to our knowledge.
By loosening the choice of basis slightly (compared to the procedure above), one obtains a proper subset of canonical bases that is compatible with the real structure and closed under deformation of the branch points.
Precisely, for each pair of roots of $a$, choose a path $\gamma$ between them in the $\lambda$-plane that avoids the other roots of $a$ and is invariant with respect to $\lambda \mapsto \bar{\lambda}^{-1}$.
Now lift this to a path $\tilde{\gamma}$ on $\Sigma_a$, and $\tilde{\gamma} - \sigma_\ast\tilde{\gamma}$ is an $A$-cycle.
In addition we may assume that the cycle is smooth.
The two possible choices of lift give the same cycle with opposite orientation.
Since $a$ has no roots on $\bbS^1$, each pair consists of roots $\alpha$ and $\bar{\alpha}^{-1}$ exchanged by $\rho$.
Not every choice of such cycles $\{A_1,\dots,A_g\}$ leads to a basis for the $A$-cycles (the subgroup with $\rho_\ast A = - A$), but if it does, then we may complete these $A$-cycles to a canonical basis of homology by choosing $B$-cycles ($A_i \cdot A_j = B_i \cdot B_j = 0$ and $A_i \cdot B_j = \delta_{ij}$).

\begin{definition}
\label{def:adapted basis}
We call a canonical basis that has been constructed in the above manner an \emph{adapted canonical basis of homology}.
\end{definition}

Due to the requirement that $\{A_1,\dots,A_g\}$ should be a basis for the $A$-cycles, we should prove the existence of adapted canonical bases.
To do so, we construct an example with the additional property that the paths $\gamma_i$ are simple and disjoint.
Due to their simplicity, no path $\gamma_i$ disconnects the $\lambda$-plane.
Thus if we choose the paths iteratively $i=1,\dots,g$ we can ensure that the next path $\gamma_i$ is disjoint from the previous ones $\gamma_1,\dots,\gamma_{i-1}$.
To see this gives a generating set of the $A$-cycles, consider a cycle $C_i$ that is the lift of a path between $\lambda=0$ and a root $\alpha_i$ of $a$ that otherwise avoids the roots of $a$ and the paths $\{\gamma_i\}$.
Then the intersection number of $C_i$ with $A_j$ is $\pm 1$ for $i=j$ and zero otherwise.
Firstly this shows that the $\{A_i\}$ span the $A$-cycles as a $\bbQ$-vector space.
Moreover if $\sum q_i A_i$ with $q_i\in\bbQ$ is an $A$-cycle, then taking the intersection with $C_j$ shows that $q_j$ is an integer.
It follows that $\{A_1,\dots,A_g\}$ generate the subgroup of $A$-cycles.

The next task is to determine how two adapted canonical bases are related.
\begin{lemma}
\label{lem:CA adapted basis}
Consider the group homomorphism $\pi_2 : \GL(g,\bbZ) \to \GL(g,\bbF_2)$ defined by taking the entries modulo 2 and let $\Perm(g,\bbF_2)$ be the permutation matrices.
Define
\[
\CA(g) = \pi_2^{-1}[\Perm(g,\bbF_2)].
\]
The name $\CA$ abbreviates ``change of $A$-cycles''.
The change of basis matrices between two adapted canonical bases are exactly those of the form~\eqref{eq:symplectic transformation} with $X\in \CA(g)$.
\end{lemma}

Before we give a proof, let us explain the group $\CA(g)$ a little.
Subgroups of integer matrix groups that are preimages of modulo reduction are called congruence subgroups.
The group $\CA(g)$ contains the principal congruence subgroup of level 2, denoted $\Gamma(2) = \ker \pi_2$.
Indeed we have the short exact sequence $1 \to \Gamma(2) \to \CA(g) \to \Perm(g,\bbF_2) \to 1$ and the semidirect product $\CA(g) = \Perm(g,\bbZ) \ltimes \Gamma(2)$, explicitly
\[
\CA(g)
= \left\{ P + 2Q \in \GL(g,\bbZ) \mid P \in \Perm(g,\bbZ), Q \in \Mat(g,\bbZ) \right\}.
\]
This form makes clear that if $X \in \CA(g)$ then so too is $X^T$, so the placement of $X$ and $(X^T)^{-1}$ in~\eqref{eq:symplectic transformation} is only a matter of convenience; both block diagonal entries belong to $\CA(g)$.
Further, just as invertible integer matrices are generated by elementary matrices and diagonal matrices, so too is $\CA(g)$ generated by ``even elementary matrices''.

\begin{lemma}
\label{lem:CA generators}
Let $E_{ij} \in \Mat(g,\bbZ)$ be the elementary matrix with $1$ in the $(i,j)$-entry and otherwise zero. 
$\Gamma(2) \subset \GL(g,\bbZ)$ is generated by the matrices $I + 2E_{ij}$ ($i \neq j$) and the diagonal matrices with entries $\pm 1$.
Thus these matrices together with permutation matrices generate $\CA(g)$.
\end{lemma}
\begin{proof}
The ``even elementary'' matrices $I + 2E_{ij}$ are invertible and belong to $\Gamma(2)$.
It is well known that multiplication with elementary matrices from the left/right corresponds to elementary row/column operations.
Likewise, multiplying from the left by $(I+2E_{ij})^k$ for example adds $2k$ times row $j$ to row $i$, while preserving $\Gamma(2)$.
In this proof we will describe how to reduce a matrix to diagonal form using even elementary row and column operations.

Let $X \in \Gamma(2)$, so its diagonal entries are odd and its off-diagonal entries are even.
Consider the entries $x_{11}$ (odd) and $x_{i1}$ (even) for some $i \geq 2$.
Call the larger in absolute value $x$ and the smaller $y$; there is at least one common element of $\{- |x|+1, - |x|+2, \dots, |x| - 1\}$ (a set with $2|x| - 1$ elements) and $x + 2|y|\bbZ$, since $|y| < |x|$ is strict due to parity.
Thus there is an even elementary row operation that reduces the larger entry in absolute value.
Performing this process iteratively is a variant of the Euclidean algorithm, and must terminate with $x_{i1} = 0$.
We can repeat this for every pair of the first and $i$-th row, so that the first column becomes $(\tilde{x}_{11}, 0, \dots, 0)^T$.
Repeating this with column operations, $X$ is reduced to the block diagonal form $X = \diag(\pm 1, \tilde{X})$ for $\tilde{X} \in \Gamma(2) \subset \GL(g-1,\bbZ)$.
The top left entry must be $\pm 1$ because $X \in \GL(g,\bbZ)$.
By induction we reduce $X$ to a diagonal matrix with entries $\pm 1$.
\end{proof}

\begin{proof}[Proof of Lemma~\ref{lem:CA adapted basis}]
First, two adapted canonical bases are canonical, so the change of basis is represented by a symplectic matrix
\[
\begin{pmatrix}
W & Y \\ Z & X
\end{pmatrix}
\in \Sp(2g, \bbZ).
\]
Next observe that the span of the $A$-cycles is preserved by a change of basis since they span the same subgroup.
Hence $Y=0$, and by the properties of symplectic matrices $W,X \in \GL(g,\bbZ)$ and $W = (X^T)^{-1}$.
We understand $W$ as the change of $A$-cycles. 
One choice of the change of basis is simply to relabel the roots of $a$ in a different order and $W$ is a permutation matrix in $\Perm(g,\bbZ)$.
Let us for the moment fix a labeling of the roots of $a$.

Let $\gamma_i, \gamma'_i$ be two paths between the pair of roots $\alpha_i$ and $\bar{\alpha}_i^{-1}$ in the $\lambda$-plane that are invariant under $\lambda \mapsto \bar{\lambda}^{-1}$.
We lift them to cycles $A_i,A'_i$.
Concatenating $\gamma'_i$ with the reverse of $\gamma_i$ gives a loop in the $\lambda$-plane.
This loop can be lifted to a cycle $C$ in $\Sigma_a$ with $A'_i - A_i = C - \sigma_\ast C = 2C$.
We already know that $(A'_i - A_i) \cdot A_j = 0$ because these are all $A$-cycles.
Therefore $A'_i = A_i + \sum_j 2 (C \cdot B_j) A_j$ and in particular $W \in I + 2\Mat(g,\bbZ)$.

Conversely, given an adapted canonical basis and $W \in \CA(g)$ we can construct cycles $A'_i$ for $i = 1,\ldots,g$ that are the lifts of $\gamma'_i$ and obey $A'_i = \sum_j W_{ij} A_j$.
Because of Lemma~\ref{lem:CA generators}, it suffices to give this construction for known generators of $\CA(g)$.
First, if $W = P$ is a permutation matrix then set $A'_i = \sum_j P_{ij} A_j$.
Second, if $W = \diag(\varepsilon_1,\dots,\varepsilon_g)$ is a diagonal matrix with entries $\pm 1$ then set $A'_i = \varepsilon_i A_i$ by reversing the orientation of the cycles as appropriate.

Therefore it remains to describe how to choose the new cycles in the case that $W$ is an even elementary matrix $I + 2E_{ij}$ or its inverse $I - 2E_{ij}$.
Let the sign be denoted $\varepsilon$.
For $k \neq i$, the cycles are unchanged and we set $A'_k = A_k$.
Let us describe a path $\tilde{\gamma}'$ in $\Sigma_a$ from $\alpha_i$ to $\bar{\alpha}_i^{-1}$.
Begin at $\alpha_i$.
Because $\gamma_i$ connects $\alpha_i$ with $\bar{\alpha}_i^{-1}$ it must cross the unit circle at least once.
Therefore we follow its lift $\tilde{\gamma}_i$ until it is over the unit circle.
Traverse an arc over the unit circle to an intersection point with $A_j$.
Then follow $\varepsilon A_j$ according to the sign in $I + 2\varepsilon E_{ij}$, returning back to the intersection point with the unit circle.
Next retrace the arc over the unit circle.
Finally follow along the remainder of $\tilde{\gamma}_i$ to $\bar{\alpha}_i^{-1}$.
Let the path $\gamma'$ in the $\lambda$-plane be the projection of $\tilde{\gamma}'$.
Since it is composed of $\gamma_i$, $\gamma_j$, and arcs on the unit circle, it is invariant under $\lambda \mapsto \bar{\lambda}^{-1}$.
According to the procedure to construct an adapted canonical basis, we should consider the cycle $\tilde{\gamma}' - \sigma_\ast \tilde{\gamma}'$.
In homology, the pairs of arcs of this cycle that lie over the unit circle cancel out, and 
\[
A'_i 
= \tilde{\gamma}_i + \varepsilon A_j - \sigma_\ast \left(\tilde{\gamma}_i + \varepsilon A_j \right)
= (\tilde{\gamma}_i - \sigma_\ast\tilde{\gamma}_i ) + \varepsilon A_j - \varepsilon \sigma_\ast A_j
= A_i + 2 \varepsilon A_j.
\] 
This shows that $W \in \CA(g)$ exactly describes the possible change of $A$-cycles.
The only condition on the choice of $B$-cycles is that they form a canonical basis, hence any choice of $Z$ such that the block matrix lies in $\Sp(2g,\bbZ)$ is acceptable, which is the symmetric matrix condition in~\eqref{eq:symplectic transformation}.
\end{proof}

Given any $a \in \calH^g_0$ we can compute $\BPer_a$ with respect to an adapted canonical basis.
By the same argument as the introduction of this section, the $\CA(g)$-orbit of $\BPer_a[\calB_a] \in \Gr(2,\bbR^g)$ is then a well-defined invariant of $\calH^g_0(a)$.
Therefore we should try to understand the corresponding quotient space.
The following theorem shows that the $\CA(g)$-orbits of planes $\Pi \in \Gr(2,\bbQ^g)$ are labeled by simple combinatorial data.

\begin{theorem}
\label{thm:CAg orbits of Pi}
Let $\Pi \in \Gr(2, \bbQ^g)$, so that $\Lambda := \Pi \cap \bbZ^g$ is a rank 2 lattice.
The $\CA(g)$-orbits of such planes $\Pi$ are in one-to-one correspondence with nonnegative integer triples $(n_1,n_2,n_3)$ with $n_1 \geq n_2 \geq n_3$, at least two entries nonzero ($n_2 \geq 1$), and $n_1 + n_2 + n_3 \leq g$.
\end{theorem}
\begin{proof}
By abuse of notation, we use $\pi_2$ for any reduction modulo $2$, not just for the homomorphism $\pi_2 : \GL(g,\bbZ) \to \GL(g, \bbF_2)$.
Choose an ordered basis $\{v_1,v_2\}$ of $\Lambda$.
Because $\Lambda$ is primitive in $\bbZ^g$ (not contained in any larger lattice of the same rank), the projection $\{\pi_2(v_1), \pi_2(v_2)\}$ is a basis for $\pi_2[\Lambda] \subset \bbF_2^g$, which is a $\dim_{\bbF_2} = 2$ subspace.
For the $i$-th coordinate of $\bbF_2^g$ consider the ``binary coordinate vector'' $V_i = (\pi_2(v_{1,i}), \pi_2(v_{2,i}))$.
It has one of four values $(0,0), (1,0), (0,1), (1,1)$.
For $i=1,\dots,g$, tally how often each of the three nonzero values occurs, and record this as an unordered triple $(n_1,n_2,n_3)$.
Choosing a different basis of $\Lambda$ acts by permutation on its three nonzero elements $\pi_2(v_1),\pi_2(v_2),\pi_2(v_1+v_2) \in \pi_2[\Lambda]$, and so permutes the binary coordinate vectors between the three nonzero values.
The action of $\Gamma(2)$ on $\bbZ^g$ leaves $\pi_2[\Lambda]$ unchanged, while the action of $\Perm(g,\bbZ)$ permutes the coordinate indices $i = 1,\dots,g$.
Thus the tally is invariant under the action of $\CA(g)$ and independent of the choice of basis of $\pi_2[\Lambda]$: it is an invariant of the $\CA(g)$-orbit of $\Lambda$.

To see the conditions in the statement of the theorem, we may take the unordered triple and arrange the elements in descending order for uniqueness.
As we really have a plane $\pi_2[\Lambda]$ (not a line or $\{0\}$), there must be at least two distinct nonzero binary coordinate vectors.
Lastly the condition $n_1+n_2+n_3 \leq g$ is simply that there are $g$ binary coordinate vectors, but some may be zero.

This construction is surjective onto the set of triples, since for a given triple $(n_1,n_2,n_3)$, we may consider the lattice $\Lambda$ spanned by 
\begin{equation}
\label{eq:basis Lambda}
\begin{aligned}
v_1 &= \bigl( 0, \dots, 0, 1, \dots, 1, 1, \dots, 1, 0, \dots, 0 \bigr)^T, \\
v_2 &= \bigl( \underbrace{1, \dots, 1}_{n_1}, \underbrace{0, \dots, 0}_{n_2}, \underbrace{1, \dots, 1}_{n_3}, 0, \dots, 0 \bigr)^T.
\end{aligned}
\end{equation}
(The convention is that $v_i$ has zeros in the $i$-th block.
This permits certain descriptions below to be more succinct.)
The projection is a basis of $\pi_2[\Lambda]$, with $0$ and $1$ now considered as elements of $\bbF_2$.
We tally $n_1$ coordinate vectors $(v_{1,i},v_{2,i})$ equal to $(0,1)$, $n_2$ equal to $(1,0)$, and $n_3$ equal to $(1,1)$.

Finally, we show that each triple $(n_1,n_2,n_3)$ corresponds to a unique orbit, by producing an element of $\CA(g)$ that takes $\Pi$ to the above example.
This is essentially an argument about lifting, which will be used again later.
Choose a basis $w_1,w_2$ for $\Lambda = \Pi \cap \bbZ^g$.
The changes of basis $w_1 \mapsto w_1, w_2 \mapsto w_1+w_2$, etc., show that 
\[
\pi_2 : \GL(\Lambda) \to \GL(\pi_2[\Lambda]) \cong \Sym(3)
\]
is surjective.
By a combination of changing the basis of $\Lambda$ and permuting the coordinates of $\bbF_2^g$, we may assume that $\pi_2(w_i) = \pi_2(v_i)$ for $i=1,2$, where $v_i$ are chosen with the same triple $(n_1,n_2,n_3)$.
Now complete $v_1,v_2$ and $w_1,w_2$ to bases of $\bbZ^g$, and let $X \in \GL(g,\bbZ)$ take one to the other.
Then modulo $2$, with respect to the basis $\{\pi_2(w_j)\}$ of $\bbF_2^g$, we write
\[
\pi_2(X)
= \begin{pmatrix}
I_2 & \tilde{B} \\
0 & \tilde{D}
\end{pmatrix},
\]
for $\tilde{D} \in \GL(g-2,\bbF_2)$.
Choose a lift $D \in \GL(g-2,\bbZ)$ of $\tilde{D}$ and any lift $B$ of $\tilde{B}$, and define
\[
Y = \begin{pmatrix}
I_2 & B \\
0 & D
\end{pmatrix} \in \GL(g,\bbZ)
\]
with respect to the basis $\{w_j\}$, so that $\pi_2(Y) = \pi_2(X)$.
By construction $\pi_2(XY^{-1}) = \pi_2(X)\pi_2(Y)^{-1} = I$, thus $XY^{-1} \in \Gamma(2)$.
The action of $XY^{-1}$ on $\Lambda$, namely $XY^{-1}(w_j) = X(w_j) = v_j$ for $j=1,2$, together with the earlier use of $\Perm(g,\bbZ)$ to permute the coordinates, proves that $\Lambda$ (and hence $\Pi$) is in the same $\CA(g)$-orbit as the standard example.
\end{proof}

There is a simple recursion to count the number of triples.
Let $c(n)$ be the number of integer triples with $n_1 \geq n_2 \geq n_3 \geq 0$, $n_2 \geq 1$ and $n_1 + n_2 + n_3 = n$.
Clearly for $n=2$ the only triple is $(1,1,0)$.
For $n \geq 3$ there are three cases to consider.
If $n_3 = 0$, then there are $\lfloor \frac{n}{2} \rfloor$ triples.
If $n_3 = n_2 = 1$, then the only triple is $(n-2,1,1)$.
If $n_3 \geq 1$ and $n_2 \geq 2$, then for any triple we may subtract one from each entry to obtain a valid triple with the sum reduced by $3$.
There are $c(n-3)$ such triples.
Hence we have the recursion
\[
c(n) 
= \left\lfloor \frac{n}{2} \right\rfloor + 1 + c(n-3)
= n + c(n-6).
\]
The number of $\CA(g)$-orbits of the doubly periodic case is therefore 
\[
\# (\Gr(2,\bbQ^g) / \CA(g) )
= \sum_{n=2}^g c(n)
= \left\lfloor\frac{1}{72}(2g^3 + 21g^2 - 6g)\right\rfloor.
\]
The small values are
\[
\begin{array}{c|ccccccccc}
  g  & 2 & 3 & 4 & 5 & 6 & 7 & 8 & 9 & 10 \\ \hline
  \# & 1 & 3 & 6 & 10 & 16 & 23 & 32 & 43 &56
\end{array}
\]
and further information can be found in~\cite{oeis:A034198}.
This is a marked improvement over the result that $\Gr(2,\bbQ^g)$ is a single $\GL(g,\bbZ)$-orbit.

The above discussion concerns the plane $\Pi=\BPer_a[\calB_a]$, but so far has ignored the complex structure $J$ that this plane also carries.
As in the introduction of this section, let us first work through the case using $\GL(g,\bbZ)$ to build familiarity.
We will again restrict to the doubly periodic case, where the quotient is tractable.
Then there is only one orbit of $\Pi \in \Gr(2,\bbQ^g)$, so fix $\Pi$ without loss of generality.
The complex structures on $\Pi$ are equivalent to the complex lines $\{ w - \ci Jw \} \subset \Pi \otimes \bbC$ that contain no purely real points, see the discussion before Definition~\ref{def:isoperiodic sets}.
As $\Pi \otimes \bbC \cong \bbC^2$, these complex lines are the points of $\CP^1 \setminus \RP^1 = \bbC \setminus \bbR$.
A more concrete way to realize this correspondence is to take an ordered basis $\{w_1,w_2\}$ of $\Lambda = \Pi \cap \bbZ^g$.
Then
\begin{equation}
\label{eq:def tau}
\tau = \frac{\varphi(w_2)}{\varphi(w_1)} \in \bbC\setminus\bbR
\end{equation}
is independent of the choice of isomorphism $\varphi : (\Pi,J) \to \bbC$ of complex vector spaces, since any two isomorphisms differ by a scalar in $\bbC^\times$.

It is important to understand how the stabilizer of $\Pi$ acts on $\Pi$. 
By extending a basis $\{w_1,w_2\}$ of $\Lambda = \Pi \cap \bbZ^g$ to a basis of $\bbZ^g$, an element of $\GL(g,\bbZ)$ that preserves $\Lambda$ has block form
\[
X 
= \begin{pmatrix}
A & B \\ 0 & D
\end{pmatrix}
\quad\text{ with }
A \in \GL(2,\bbZ), D \in \GL(g-2,\bbZ).
\]
In particular, it acts on $\Lambda$ as $A$.
Conversely, given $A \in \GL(2,\bbZ)$ we may set $B = 0, D = I$, which shows that the $\GL(g,\bbZ)$-stabilizer of $\Lambda$ acts on $\Lambda$ through all of $\GL(2,\bbZ)$.
This induces a $\GL(2,\bbZ)$ action on $\tau$ by Möbius transformations, with $\begin{pmatrix}
1 & 0 \\ 0 & -1
\end{pmatrix}$ sending $\tau \mapsto -\tau$ and $-I$ fixing $\tau$.
Therefore 
\[
(\bbC \setminus \bbR) / \GL(2,\bbZ)
= \bbH  / \SL(2,\bbZ)
= \bbH  / \PSL(2,\bbZ),
\]
where $\bbH$ denotes the upper half-plane.
This quotient is familiar to us as the moduli space of rank-2 lattices in $\bbC$ up to homothety.
The connection to the present situation is that the classical problem fixes the complex structure and varies the lattice, whereas we have fixed the lattice and varied the complex structure.
This quotient has a standard fundamental domain
\[
\calF_1 = \left\{ x + \ci y \in \bbH \mid -\tfrac{1}{2} \leq x \leq \tfrac{1}{2}, x^2 + y^2 \geq 1 \right\}
\]
with boundary gluing $(-\tfrac{1}{2},y) \sim (\tfrac{1}{2},y)$ and $e^{\ci\phi} \sim e^{\ci (\pi - \phi)}$.
In summary, the $\GL(2,\bbZ)$-orbit of $\tau$ is constant along $\calH^g_0(a,\tau)$, justifying the notation introduced in Definition~\ref{def:isoperiodic sets}, and the orbits are parametrized by $\calF_1$.

Now that we understand the $\GL(g,\bbZ)$ picture, we can carry through the same idea with an adapted canonical basis and the group $\CA(g)$.
The key is to understand how the $\CA(g)$-stabilizer of $\Lambda$ acts on $\Lambda$.
Interestingly, the answer will depend on the triple $(n_1,n_2,n_3)$.
Let us consider an element $X \in \CA(g)$ that preserves $\Lambda$.
By definition $\pi_2(X) \in \Perm(g,\bbF_2)$ and it preserves $\pi_2[\Lambda]$.
Let us consider which permutations preserve $\pi_2[\Lambda]$.
If we take the standard form of $\Lambda$ from~\eqref{eq:basis Lambda}, then the three nonzero elements of $\pi_2[\Lambda] \subset \bbF_2^g$ are 
\begin{align*}
\pi_2(v_1) &= \bigl( 0, \dots, 0, 1, \dots, 1, 1, \dots, 1, 0, \dots, 0 \bigr)^T, \\
\pi_2(v_2) &= \bigl( 1, \dots, 1, 0, \dots, 0, 1, \dots, 1, 0, \dots, 0 \bigr)^T, \\
\pi_2(v_3) = \pi_2(v_1 + v_2) &= \bigl( \underbrace{1, \dots, 1}_{n_1}, \underbrace{1, \dots, 1}_{n_2}, \underbrace{0, \dots, 0}_{n_3}, 0, \dots, 0 \bigr)^T.
\end{align*}
The group $\GL(\pi_2[\Lambda])$ is exactly the symmetric group $\Sym(3)$ of these three nonzero elements.
It is only possible to permute these elements by the action of $\Perm(g,\bbF_2)$, that is, a permutation of the coordinates, if some of the $n_i$ are equal to one another.
For example, if $n_1 = n_2$ then there is a coordinate permutation of $\bbF_2^g$ that exchanges $\pi_2(v_1)$ and $\pi_2(v_2)$ while fixing $\pi_2(v_3)$.
The general rule is that if $n_i = n_j$ then we can swap $v_i$ and $v_j$ while fixing the third element.
Define $H$ as the subgroup of $\GL(\pi_2[\Lambda])$ induced by the action of $\Perm(g,\bbF_2)$.
By our reasoning, if $n_1=n_2=n_3$ then $H = \GL(\pi_2[\Lambda])$; if all are distinct then $H = \{\id\}$; if $n_i = n_j \neq n_k$ then $H = \{\id, (v_i, v_j)\} \cong \Sym(2)$.
These are the only possibilities.

\begin{lemma}
\label{lem:CA stabilizer}
The image of the $\CA(g)$-stabilizer of $\Lambda$ inside $\GL(\Lambda)$ is $\pi_2^{-1}[H]$.
\end{lemma}
\begin{proof}
We have argued above that if $X \in \CA(g)$ stabilizes $\Lambda$ then $\pi_2(X)|_{\Lambda} \in H$.
It remains to prove the converse.
Suppose $A \in \pi_2^{-1}[H] \subset \GL(\Lambda)$ and let $h = \pi_2(A) \in \GL(\pi_2[\Lambda])$ and let $P \in \Perm(g, \bbF_2)$ be the corresponding permutation matrix.
Then we may find an element $X \in \GL(g,\bbZ)$ with $X|_{\Lambda} = A$ and $\pi_2(X) = P$, by the lifting argument contained in the proof of Theorem~\ref{thm:CAg orbits of Pi}.
Since $P$ is a permutation matrix, we conclude that $X \in \CA(g)$ and thus $A$ is in the image of the $\CA(g)$-stabilizer of $\Lambda$.
\end{proof}

To finish the argument, we examine the quotient of $\bbC\setminus\bbR$ under the action of the $\CA(g)$-stabilizer of $\Lambda$.
If $n_1=n_2=n_3$, then $H$ is all of $\GL(\pi_2[\Lambda])$ and lying in its preimage is no condition at all.
Hence the stabilizer acts by the full group $\GL(\Lambda)$ and the quotient space is $\calF_1$, as above.
At the other extreme, if all $n_i$ are distinct then the stabilizer acts by the principal congruence subgroup $\Gamma(2)$.
In this case, to compute $\tau$ in~\eqref{eq:def tau} we choose the basis $\{v_1,v_2\}$ from~\eqref{eq:basis Lambda}.
Then, after applying $\Gamma(2)$ there is a unique point in a fundamental domain
\[
\calF_6 
:= \{ x + \ci y \in \bbH \mid -1 \leq x \leq 1, (x-\tfrac{1}{2})^2 + y^2 \geq \tfrac{1}{4}, (x+\tfrac{1}{2})^2 + y^2 \geq \tfrac{1}{4}\}
\]
with boundary gluing $-1 + \ci y \sim 1 + \ci y$ and $-\frac{1}{2} + \frac{1}{2}e^{\ci \phi} \sim \frac{1}{2} + \frac{1}{2}e^{\ci (\pi - \phi)}$, as described in~\cite[Fig.~7]{Farkas2001}; see Figure~\ref{fig:F6 and F3}.

\begin{figure}
\centering
\begin{subfigure}{0.48\textwidth}
\centering
\begin{tikzpicture}[scale=2.2]
\fill[gray!15]
  (-1,0) arc[start angle=180, end angle=0, radius=0.5]
         arc[start angle=180, end angle=0, radius=0.5]
  -- (1,1.7) -- (-1,1.7) -- cycle;

\draw[gray] (-1.3,0) -- (1.3,0);

\draw[gray, dashed] (1,0) arc[start angle=0, end angle=180, radius=1];

\draw[thick, mid arrow] (-1,0) -- (-1,1.5);
\draw[thick, mid arrow] ( 1,0) -- ( 1,1.5);
\draw[thick, dotted] (-1,1.5) -- (-1,1.7);
\draw[thick, dotted] ( 1,1.5) -- ( 1,1.7);

\draw[thick, mid arrow=0.5, mid arrow=0.58]
  (-1,0) arc[start angle=180, end angle=0, radius=0.5];
\draw[thick, mid arrow=0.5, mid arrow=0.58]
  ( 1,0) arc[start angle=0, end angle=180, radius=0.5];

\foreach \x in {-1,0,1} {
  \filldraw[fill=white] (\x,0) circle (0.025);
}
\node[below] at (-1,0) {$-1$};
\node[below] at ( 0,0) {$0$};
\node[below] at ( 1,0) {$1$};
\fill (0,1) circle (0.02) node[above right] {$\ci$};
\node at (0,1.35) {$\calF_6$};
\end{tikzpicture}
\end{subfigure}
\hfill
\begin{subfigure}{0.48\textwidth}
\centering
\begin{tikzpicture}[scale=2.2]
\fill[gray!15]
  (1,0) arc[start angle=0, end angle=180, radius=1]
  -- (-1,1.7) -- (1,1.7) -- cycle;

\draw[gray] (-1.3,0) -- (1.3,0);

\draw[thick, mid arrow] (-1,0) -- (-1,1.5);
\draw[thick, mid arrow] ( 1,0) -- ( 1,1.5);
\draw[thick, dotted] (-1,1.5) -- (-1,1.7);
\draw[thick, dotted] ( 1,1.5) -- ( 1,1.7);

\draw[thick, mid arrow=0.5, mid arrow=0.58]
  ( 1,0) arc[start angle=0, end angle=90, radius=1];
\draw[thick, mid arrow=0.5, mid arrow=0.58]
  (-1,0) arc[start angle=180, end angle=90, radius=1];

\foreach \x in {-1,1} {
  \filldraw[fill=white] (\x,0) circle (0.025);
}
\node[below] at (-1,0) {$-1$};
\node[below] at ( 0,0) {$0$};
\node[below] at ( 1,0) {$1$};
\fill (0,1) circle (0.02) node[above] {$\ci$};
\node at (0,1.35) {$\calF_3$};
\end{tikzpicture}
\end{subfigure}
\caption{
The fundamental domains $\calF_6$ and $\calF_3$.
Boundary segments with matching arrows are glued.
The dashed unit circle in $\calF_6$ divides it into two halves, which are exchanged by $\tau \mapsto -\tau^{-1}$.
}
\label{fig:F6 and F3}
\end{figure}

In the final case, where exactly two of $(n_1,n_2,n_3)$ are equal, say $n_i=n_j$, we should choose a basis $\{w_1,w_2\}$ so that $w_1 \leftrightarrow w_2$ belongs to the stabilizer.
This is the case if $\pi_2(w_1) = \pi_2(v_i)$ and $\pi_2(w_2) = \pi_2(v_j)$, for the standard basis~\eqref{eq:basis Lambda}.
Then $H$ is the permutation group of the basis, and hence $\pi_2^{-1}[H]$ is isomorphic to $\CA(2)$.
In addition to the action of $\Gamma(2)$ on $\tau$, we have the swap of basis vectors, $\tau \mapsto \tau^{-1}$.
We can compose this with negation to stay in the upper half-plane; the composition gives $\tau \mapsto - \tau^{-1}$.
This is a bijection between the parts of $\calF_6$ inside and outside the unit circle.
Therefore we have a fundamental domain
\[
\calF_3 
:= \{ x + \ci y \in \bbH \mid -1 \leq x \leq 1, x^2 + y^2 \geq 1\}
\]
with boundary gluing $-1 + \ci y \sim 1 + \ci y$ and $e^{\ci\phi} \sim e^{\ci (\pi - \phi)}$; see Figure~\ref{fig:F6 and F3}.
The fundamental domains $\calF_3$ and $\calF_6$ fit in the schema described in~\cite[p.~156]{Newman1972}.
We summarize the discussion in the following theorem.
\begin{theorem}
\label{thm:CAg orbits of τ}
Fix $\Pi \in \Gr(2, \bbQ^g)$ with invariant $(n_1,n_2,n_3)$ as per Theorem~\ref{thm:CAg orbits of Pi}.
A complex structure on $\Pi$ defines an invariant $\tau$ by~\eqref{eq:def tau}.
The complex structures on $\Pi$ up to the action of the $\CA(g)$-stabilizer are in one-to-one correspondence with 
\[
\tau \in \begin{cases}
\calF_1 &\text{if } n_1=n_2=n_3, \\
\calF_3 &\text{if } n_1=n_2> n_3 \text{ or } n_1 > n_2=n_3, \\
\calF_6 &\text{if }n_1>n_2>n_3.
\end{cases}
\]
\end{theorem}

\begin{example}
For genus $g=2$, there is only one possible triple, namely $(1,1,0)$.
The corresponding moduli space for $\tau$ is $\calF_3$.
For genus $g=3$, there are three possible triples, namely $(2,1,0)$, $(1,1,1)$ and $(1,1,0)$.
The corresponding moduli spaces for $\tau$ are $\calF_6$, $\calF_1$, and $\calF_3$ respectively.
\end{example}

\bibliographystyle{alpha}
\bibliography{cmc_moduli_gg_g2.bib}

@article{Bobenko1991,
  title = {All constant mean curvature tori in $\mathbf{R}^3,\;{S}^3,\;{H}^3$ in terms of theta-functions},
  author = {Bobenko, Alexander I.},
  year = 1991,
  journal = {Mathematische Annalen},
  volume = {290},
  number = {2},
  pages = {209-245},
  doi = {10.1007/BF01459243}
}

@inproceedings{Carberry2013a,
  title = {Harmonic Maps and Integrable Systems},
  booktitle = {Geometry and topology down under},
  author = {Carberry, Emma},
  year = 2013,
  series = {Contemporary Mathematics},
  volume = {597},
  pages = {139-163},
  publisher = {American Mathematical Society},
  address = {Providence, RI},
  doi = {10.1090/conm/597/11770},
  url = {http://www.ams.org/conm/597},
  urldate = {2019-07-14},
  isbn = {978-0-8218-8480-5 978-1-4704-1025-4},
  langid = {english}
}

@article{Carberry2016,
  title = {The prevalence of tori amongst constant mean curvature planes in $\mathbb{R}^3$},
  author = {Carberry, Emma and Schmidt, Martin Ulrich},
  year = {2016},
  month = aug,
  journal = {Journal of Geometry and Physics},
  volume = {106},
  pages = {352-366},
  issn = {03930440},
  doi = {10.1016/j.geomphys.2016.04.010},
  url = {http://linkinghub.elsevier.com/retrieve/pii/S0393044016300961},
  urldate = {2018-01-10},
  langid = {english}
}

@article{Carberry2016a,
  title = {The closure of spectral data for constant mean curvature tori in ${S}^3$},
  author = {Carberry, Emma and Schmidt, Martin Ulrich},
  year = {2016},
  month = jan,
  journal = {Journal für die Reine und Angewandte Mathematik (Crelle's Journal)},
  volume = {2016},
  number = {721},
  pages = {149-166},
  issn = {0075-4102, 1435-5345},
  doi = {10.1515/crelle-2014-0068},
  url = {https://www.degruyter.com/view/j/crelle.2016.2016.issue-721/crelle-2014-0068/crelle-2014-0068.xml},
  urldate = {2018-01-09}
}

@article{Carberry2019,
  title = {Whitham deformations and the space of harmonic tori in $\mathbb{S}^3$},
  author = {Carberry, Emma and Ogilvie, Ross},
  year = 2019,
  month = jun,
  journal = {Journal of the London Mathematical Society},
  volume = {99},
  number = {3},
  pages = {945-964},
  issn = {0024-6107, 1469-7750},
  doi = {10.1112/jlms.12191},
  url = {https://onlinelibrary.wiley.com/doi/abs/10.1112/jlms.12191},
  urldate = {2019-07-14},
  langid = {english}
}

@article{Carberry2020,
  title = {The Space of Equivariant Harmonic Tori in the 3-Sphere},
  author = {Carberry, Emma and Ogilvie, Ross},
  year = {2020},
  journal = {Journal of Geometry and Physics},
  volume = {157},
  pages = {103808},
  issn = {0393-0440},
  doi = {10.1016/j.geomphys.2020.103808},
  url = {https://www.sciencedirect.com/science/article/pii/S0393044020301509}
}

@misc{Carberry2023arxiv,
  title = {The Space of Genus Two Spectral Curves of Constant Mean Curvature Tori in $\mathbb{R}^3$},
  author = {Carberry, Emma and Kilian, Martin and Klein, Sebastian and Schmidt, Martin Ulrich},
  year = 2023,
  month = mar,
  number = {arXiv:2110.00436},
  eprint = {2110.00436},
  primaryclass = {math},
  publisher = {arXiv},
  doi = {10.48550/arXiv.2110.00436},
  url = {http://arxiv.org/abs/2110.00436},
  urldate = {2025-05-28},
  archiveprefix = {arXiv},
  note = {arXiv:2110.00436},
}

@misc{Carberry2025arxiv,
  title = {Universal Deformations of a Curve and a Differential},
  author = {Carberry, Emma and Schmidt, Martin Ulrich},
  year = 2025,
  month = nov,
  number = {arXiv:2211.11442},
  eprint = {2211.11442},
  primaryclass = {math.AG},
  publisher = {arXiv},
  doi = {10.48550/arXiv.2211.11442},
  url = {http://arxiv.org/abs/2211.11442},
  urldate = {2026-10-12},
  archiveprefix = {arXiv},
  note = {arXiv:2211.11442},
}

@incollection{Ercolani1993,
  title = {Hyperelliptic Curves That Generate Constant Mean Curvature Tori in {{R}}{$^{3}$}},
  booktitle = {Integrable Systems: The {{Verdier}} Memorial Conference. {{Actes}} Du Colloque International de {{Luminy}}, {{France}}, {{July}} 1-5, 1991},
  author = {Ercolani, N. M. and Kn{\"o}rrer, H. and Trubowitz, E.},
  year = 1993,
  pages = {81--114},
  publisher = {Birkh\"auser},
  address = {Boston, MA},
  isbn = {0-8176-3653-6},
  langid = {english},
  zbmath = {732107},
  zmnumber = {0867.35082}
}

@book{Farkas2001,
  title = {Theta {{Constants}}, {{Riemann Surfaces}}, and the {{Modular Group}}: {{An Introduction}} with {{Applications}} to {{Uniformization Theorems}}, {{Partition Identities}}, and {{Combinatorial Number Theory}}},
  shorttitle = {Theta {{Constants}}, {{Riemann Surfaces}}, and the {{Modular Group}}},
  author = {Farkas, Hershel M. and Kra, Irwin},
  year = 2001,
  series = {Graduate {{Studies}} in {{Mathematics}}},
  number = {37},
  publisher = {American Mathematical Society},
  address = {Providence, RI},
  isbn = {978-0-8218-1392-8},
  lccn = {QA345 .F37 2001}
}

@book{Griffiths1994,
  title = {Principles of Algebraic Geometry},
  author = {Griffiths, Phillip and Harris, Joe},
  year = 1994,
  publisher = {J. Wiley \& Sons},
  address = {New York, NY},
  isbn = {978-0-471-05059-9},
  langid = {english}
}

@article{Grinevich1995,
  title = {Period Preserving Nonisospectral Flows and the Moduli Space of Periodic Solutions of Soliton Equations},
  author = {Grinevich, P. G. and Schmidt, M. U.},
  year = 1995,
  month = oct,
  journal = {Physica D: Nonlinear Phenomena},
  volume = {87},
  number = {1-4},
  pages = {73--98},
  issn = {01672789},
  doi = {10.1016/0167-2789(95)00139-U},
  url = {https://linkinghub.elsevier.com/retrieve/pii/016727899500139U},
  urldate = {2022-06-27},
  langid = {english}
}

@article{Gross1981,
  title = {Real Algebraic Curves},
  author = {Gross, Benedict H. and Harris, Joe},
  year = 1981,
  journal = {Annales scientifiques de l'\'Ecole normale sup\'erieure},
  volume = {14},
  number = {2},
  pages = {157--182},
  issn = {0012-9593, 1873-2151},
  doi = {10.24033/asens.1401},
  url = {http://www.numdam.org/item?id=ASENS_1981_4_14_2_157_0},
  urldate = {2025-06-01}
}

@book{Hardy1975,
  title = {An Introduction to the Theory of Numbers},
  author = {Hardy, Godfrey H. and Wright, Edward M.},
  year = 1975,
  series = {Oxford Mathematics},
  edition = {Fourth},
  publisher = {Oxford University Press},
  address = {Oxford},
  doi = {10.1093/oso/9780199219858.001.0001},
  isbn = {978-1-383-03595-7},
  langid = {english}
}

@article{Hauswirth2016,
  title = {Mean-Convex {{Alexandrov}} Embedded Constant Mean Curvature Tori in the 3-Sphere},
  author = {Hauswirth, L. and Kilian, M. and Schmidt, M. U.},
  year = 2016,
  month = mar,
  journal = {Proceedings of the London Mathematical Society},
  volume = {112},
  number = {3},
  pages = {588--622},
  issn = {0024-6115, 1460-244X},
  doi = {10.1112/plms/pdw002},
  url = {http://doi.wiley.com/10.1112/plms/pdw002},
  urldate = {2019-07-15},
  langid = {english}
}

@misc{Hauswirth2017arxiv,
  title = {Singularities of {{Whitham}} Flows for Hyperelliptic Spectral Curves},
  author = {Hauswirth, Laurent and Kilian, Martin and Schmidt, Martin U.},
  year = 2017,
  month = sep,
  eprint = {1709.02330},
  publisher = {arXiv},
  url = {http://arxiv.org/abs/1709.02330},
  urldate = {2019-06-24},
  archiveprefix = {arXiv},
  note = {arXiv:1709.02330},
}

@incollection{Heller2018a,
  title = {Towards a Constrained {{Willmore}} Conjecture},
  booktitle = {Willmore Energy and {{Willmore}} Conjecture},
  author = {Heller, Lynn and Pedit, Franz},
  year = 2018,
  pages = {119--138},
  publisher = {CRC Press},
  address = {Boca Raton, FL},
  isbn = {978-1-4987-4463-8 978-1-4987-4464-5},
  langid = {english},
  zbmath = {6843378},
  zmnumber = {1384.53007}
}

@article{Hitchin1990,
  title = {Harmonic Maps from a 2-Torus to the 3-Sphere},
  author = {Hitchin, Nigel J.},
  year = {1990},
  journal = {Journal of Differential Geometry},
  volume = {31},
  pages = {627-710},
  doi = {10.4310/jdg/1214444631},
  url = {http://projecteuclid.org/euclid.jdg/1214444631}
}

@book{Huybrechts2005,
  title = {Complex {{Geometry}}},
  author = {Huybrechts, Daniel},
  year = 2005,
  series = {Universitext},
  publisher = {Springer-Verlag},
  address = {Berlin, Heidelberg},
  doi = {10.1007/b137952},
  url = {http://link.springer.com/10.1007/b137952},
  urldate = {2024-04-10},
  copyright = {http://www.springer.com/tdm},
  isbn = {978-3-540-21290-4},
  langid = {english}
}

@article{Jaggy1994,
  title = {On the Classification of Constant Mean Curvature Tori in {{$\mathbb{R}^3$}}},
  author = {Jaggy, Christian},
  year = 1994,
  month = dec,
  journal = {Commentarii Mathematici Helvetici},
  volume = {69},
  number = {1},
  pages = {640--658},
  issn = {0010-2571, 1420-8946},
  doi = {10.1007/BF02564507},
  url = {https://link.springer.com/10.1007/BF02564507},
  urldate = {2026-09-28},
  langid = {english}
}

@article{Kapouleas1995,
  title = {Constant Mean Curvature Surfaces Constructed by Fusing {{Wente}} Tori},
  author = {Kapouleas, Nikolaos},
  year = 1995,
  month = dec,
  journal = {Inventiones Mathematicae},
  volume = {119},
  number = {1},
  pages = {443--518},
  issn = {0020-9910, 1432-1297},
  doi = {10/bhmwxm},
  url = {http://link.springer.com/10.1007/BF01245190},
  urldate = {2021-07-20},
  langid = {english}
}

@article{Kilian2007,
  title = {Unitarization of Monodromy Representations and Constant Mean Curvature Trinoids in 3-Dimensional Space Forms},
  author = {Kilian, M. and Kobayashi, S.-P. and Rossman, W. and Schmitt, N.},
  year = 2007,
  month = jun,
  journal = {Journal of the London Mathematical Society},
  volume = {75},
  number = {3},
  eprint = {math/0403366},
  pages = {563--581},
  issn = {00246107},
  doi = {10.1112/jlms/jdm005},
  url = {http://arxiv.org/abs/math/0403366},
  urldate = {2025-09-23},
  archiveprefix = {arXiv}
}

@article{Kilian2015,
  title = {Flows of Constant Mean Curvature Tori in the 3-Sphere: The Equivariant Case},
  author = {Kilian, Martin and Schmidt, Martin Ulrich and Schmitt, Nicholas},
  year = {2015},
  month = nov,
  journal = {Journal f{\"u}r die Reine und Angewandte Mathematik (Crelle's Journal)},
  volume = {707},
  pages = {45-86},
  issn = {14355345},
  doi = {10.1515/crelle-2013-0079},
  url = {http://arxiv.org/abs/1011.2875},
  archiveprefix = {arXiv},
  arxivid = {1011.2875}
}

@article{Knopf2025,
  title = {Solutions of the Sinh-{{Gordon}} Equation of Spectral Genus Two and Constrained {{Willmore}} Tori {{I}}},
  author = {Knopf, M. and Pe{\~n}a~Hoepner, R. and Schmidt, M. U.},
  year = 2025,
  month = jan,
  journal = {The Journal of Geometric Analysis},
  volume = {35},
  number = {1},
  pages = {29},
  issn = {1050-6926, 1559-002X},
  doi = {10.1007/s12220-024-01861-w},
  url = {https://link.springer.com/10.1007/s12220-024-01861-w},
  urldate = {2026-02-27},
  langid = {english}
}

@book{Lee2012,
  title = {Introduction to {{Smooth Manifolds}}},
  author = {Lee, John M.},
  year = 2012,
  series = {Graduate {{Texts}} in {{Mathematics}}},
  volume = {218},
  publisher = {Springer New York},
  address = {New York, NY},
  doi = {10.1007/978-1-4419-9982-5},
  url = {https://link.springer.com/10.1007/978-1-4419-9982-5},
  urldate = {2023-04-14},
  isbn = {978-1-4419-9981-8 978-1-4419-9982-5},
  langid = {english}
}

@book{Mumford1984,
  title = {Tata Lectures on {{Theta II}}},
  author = {Mumford, David},
  year = 1984,
  series = {Progress in Mathematics},
  number = {43},
  publisher = {Birkh\"auser},
  address = {Boston},
  isbn = {978-3-7643-3110-8 978-0-8176-3110-9},
  langid = {english},
  lccn = {515.984}
}

@book{Newman1972,
  title = {Integral Matrices},
  author = {Newman, Morris},
  year = 1972,
  series = {Pure and Applied Mathematics; a Series of Monographs and Textbooks},
  number = {45},
  publisher = {Academic Press},
  address = {New York},
  isbn = {978-0-12-517850-1 978-0-08-087358-9},
  langid = {english}
}

@misc{oeis:A034198,
  author       = {{OEIS Foundation Inc.}},
  title        = {Entry {A034198} in {T}he {O}n-{L}ine {E}ncyclopedia of
                  {I}nteger {S}equences},
  year         = {2026},
  howpublished = {\url{https://oeis.org/A034198}},
  note         = {Accessed 23 September 2026}
}

@article{Pinkall1989,
  title = {On the Classification of Constant Mean Curvature Tori},
  author = {Pinkall, Ulrich and Sterling, Ivan},
  year = {1989},
  journal = {Annals of Mathematics},
  volume = {130},
  number = {2},
  pages = {407-451},
  issn = {0003486X},
  doi = {10.2307/1971425},
  url = {http://dx.doi.org/10.2307/1971425}
}

@article{Wente1986,
  title = {Counterexample to a Conjecture of {{H}}. {{Hopf}}},
  author = {Wente, Henry C.},
  year = 1986,
  journal = {Pacific Journal of Mathematics},
  volume = {121},
  number = {1},
  pages = {193--243},
  issn = {0030-8730},
  doi = {10.2140/pjm.1986.121.193},
  url = {http://msp.org/pjm/1986/121-1/p18.xhtml}
}

\bigskip
{\small
Nicolas Hasse, Universit\"at Mannheim\\
\emph{Email address:} \href{mailto:nicolas.hasse@uni-mannheim.de}{\texttt{nicolas.hasse@uni-mannheim.de}}

Ross Ogilvie, Universit\"at Mannheim\\
\emph{Email address:} \href{mailto:r.ogilvie@uni-mannheim.de}{\texttt{r.ogilvie@uni-mannheim.de}}

Martin U.\ Schmidt, Universit\"at Mannheim\\
\emph{Email address:} \href{mailto:schmidt.math@uni-mannheim.de}{\texttt{schmidt.math@uni-mannheim.de}}
\par}
\end{document}